\documentclass[a4paper]{article}

\usepackage[utf8]{inputenc}
\usepackage{lmodern}
\usepackage{amssymb,amsfonts,amsthm,amsmath,bbm,mathrsfs,aliascnt,mathtools,pifont}
\usepackage[english]{babel}
\usepackage[colorlinks]{hyperref}
\usepackage{tikz}
\usetikzlibrary{decorations}
\usetikzlibrary{decorations.pathreplacing}
\tikzset{individu/.style={draw,thick}}

\theoremstyle{plain}
\newtheorem{theorem}{Theorem}[section]
\newtheorem{corollary}[theorem]{Corollary}
\newtheorem{lemma}[theorem]{Lemma}
\newtheorem{proposition}[theorem]{Proposition}

\theoremstyle{definition}

\theoremstyle{remark}
\newtheorem{remark}[theorem]{Remark}

\numberwithin{equation}{section}

\newcommand{\N}{\mathbb{N}}
\newcommand{\Z}{\mathbb{Z}}
\newcommand{\R}{\mathbb{R}}

\newcommand{\ind}[1]{\mathbf{1}_{\left\{#1\right\}}}
\newcommand{\indset}[1]{\mathbf{1}_{#1}}

\renewcommand{\bar}[1]{\overline{#1}}
\renewcommand{\tilde}[1]{\widetilde{#1}}
\renewcommand{\hat}[1]{\widehat{#1}}

\newcommand{\dd}{\mathrm{d}}

\DeclareMathOperator{\E}{\mathbb{E}}
\newcommand{\Q}{\mathbb{Q}}
\renewcommand{\P}{\mathbb{P}}

\newcommand{\GW}{{{\normalfont\text{\tiny{GW}}}}}

\newcommand{\ks}{{\mathrm{k}_*}}

\renewcommand{\rho}{\varrho}
\renewcommand{\epsilon}{\varepsilon}

\title{Reinforced Galton-Watson processes II:\\ Large time behaviors }
\author{Jean Bertoin\thanks{Institute of Mathematics, University of Zurich, Switzerland.} \and Bastien Mallein\thanks{Institut de Math\'ematiques de Toulouse,  Universit\'e de Toulouse, France.}}
\date{}
\begin{document}

\maketitle

\begin{abstract}
Reinforced Galton-Watson processes have been introduced in \cite{BM} as population models with non-overlapping generations, such that reproduction events along genealogical lines can be repeated at random. We investigate here some of their sample path properties such as asymptotic growth rates and survival, for which the effects of reinforcement on the evolution appear  quite strikingly.
\end{abstract}


\medskip


\section{Introduction and main results}
\label{sec:introduction}

Reinforcement is a fundamental concept in many sciences, including  notably Behavioral Psychology (as a key part of Skinner's behavioral theory of learning \cite{Skinner}) and Artificial Intelligence \cite{Gro}, where this terminology refers to methods that increase the likelihood of certain evolutions in both natural and artificial systems.
It has been introduced in the setting of stochastic processes by Coppersmith and Diaconis in an influential unpublished article, where these authors modified step after step
the dynamics of a random walk on a graph, in such a way
that transitions which have already been often made in the past are more likely occur again in the future.

It is a recent development in Demography that a reinforcement feature can be detected in the genealogy of human populations. In many human populations, the fertility levels of parents and children are positively correlated \cite{BeauSol,Murphy}; this may be explained e.g. by socioeconomic, cultural or inherited factors. This observation, sometimes referred to as Intergenerational Transmission of Fertility, invalidates population models in which individuals reproduce independently one from the others, and provides an incentive for the study of reinforced versions which incorporate a dependency structure of  fertility levels along lineages.

The Galton-Watson branching process is a population model in which every individual reproduces independently of the others according a probability distribution $(\nu(k))_{k \geq 0}$; any individual has probability $\nu(k)$ of having $k$ children at the next generation. This population model was originally introduced to study the evolution of human demography, namely the possible disappearance of family names. It is well-known that this process almost surely dies out if and only if the mean number of children of a given individual satisfies $m_\GW := \sum k \nu(k) < 1$. We assume throughout this work that the reproduction law has bounded support, i.e. there exists $\ks \geq 2$ with $\nu(\ks) > 0$ such that $\boldsymbol{\nu} = (\nu(k))_{0 \leq k \leq \ks}$ is a probability vector. In other words, $\ks$ is the maximal possible number of children.

We recently  introduced in \cite{BM} a  reinforced version of Galton-Watson processes that  involves random repetitions of reproduction events and depends on a parameter $q\in(0,1)$. The evolution can be depicted as follows. Each individual at any generation $n\geq 1$ picks a forebear uniformly at random on its ancestral lineage, independently of the other individuals. Then either with probability $q$,  this individual begets the same number of children as the selected forebear, or with complementary probability $1-q$,  the number of its children  is  an independent sample from the reproduction law $\boldsymbol{\nu}$. Reproduction events that occurred at early stages of the process are thus more likely to be repeated in the future, as they are common to a large number of genealogical lines, and informally speaking, the reinforced process thus  keeps some memory of its past.  Clearly, the usual Galton-Watson  evolution corresponds to the boundary case $q=0$ without reinforcement (i.e. without memory).

Let $Z=(Z(n))_{n\geq 0}$ be a reinforced Galton-Watson process  with reproduction law $\boldsymbol{\nu}$ and reinforcement parameter $q\in(0,1)$, started from a single ancestor.
More precisely,  $Z(n)$ denotes the size of the $n$-th generation, and for $k \geq 1$, we write $\P_k$  for the distribution of $Z$ conditioned on $Z(1)=k$. Let us recall a simplified  version of our main result  about the asymptotic behavior of the averaged population size  for large generations; see \cite[Theorem 7.3]{BM}. As $n\to \infty$, we have
\begin{equation}\label{E:asympmoyen}
\E_{ \ks } (Z(n)) \sim    \frac{  m_{\boldsymbol{\nu},q}^{n}}{q+
(1-q) \nu( \ks )} ,
 \end{equation}
whereas
\begin{equation}\label{E:asympmoyen0}
\E_{\ell} (Z(n)) = o\left( m_{\boldsymbol{\nu},q}^{n}\right) \quad\text{for any $\ell< \ks $ with $\nu(\ell)>0$,}
 \end{equation}
 where
\begin{equation}\label{E:I}
m_{\boldsymbol{\nu},q}\coloneqq \frac{q}{\int_0^{1/ \ks } \Pi(t) \dd t},
\end{equation}
and
\begin{equation} \label{E:pi}
\Pi(t)\coloneqq  \prod_{k:\nu(k)>0} (1-tk)^{\nu(k)(1-q) /q},  \qquad0\leq  t\leq 1/ \ks .
\end{equation}

We now present the  motivations for the present work and our two main results. Recall that for the usual Galton-Watson process $Z_\GW$ with reproduction law
$\boldsymbol{\nu}$,  not only the mean population size at the $n$-th generation is\footnote{This has to be compared and contrasted with \eqref{E:asympmoyen} and \eqref{E:asympmoyen0};  recall also from  \cite[Proposition 7.1(i)]{BM} that
$m_\GW=\lim_{q\to 0+} m_{\boldsymbol{\nu},q}$.}
$$\E(Z_\GW(n)) = m_\GW^n\qquad \text{ for all }n\geq 1, $$
but in the supercritical case where $m_\GW>1$, there is the much more precise pathwise convergence
$$\lim_{n\to \infty} m_\GW^{-n} Z_\GW(n) = W_\GW \qquad\text{a.s. and in }L^1.$$
Furthermore, in the original Galton-Watson process, the random variable $W_\GW$ vanishes exactly on the event of eventual extinction of $Z_\GW$. Our first main result shows that reinforcement may change dramatically this property.

In this direction, we need to introduce some more notation. Consider the sub-process $Z_*=(Z_*(n))_{n\geq 0}$ of the reinforced Galton-Watson process which results by suppressing every progeny (together with its descent) that has  size strictly less than $ \ks $. Since by construction, all the forebears of an individual in this sub-process have exactly $ \ks $ children,  $Z_*$ is a true Galton-Watson process  under $\P_{ \ks }$ (recall that  the latter  denotes the distribution of the reinforced Galton-Watson process when the ancestor has the maximal number $\ks$ of children).
Specifically, all individuals in this sub-process, with the exception of the ancestor, have either $ \ks $ children with probability $q+(1-q)\nu( \ks )$, or $0$ child with complementary probability, independently one of each other's. We write
\begin{equation}\label{E:m*}
m_{*,q} \coloneqq \ks (q + (1-q)\nu( \ks ))
\end{equation}
for the mean reproduction number of $Z_*$ and stress that $m_{*,q}\leq m_{\boldsymbol{\nu},q}$; see \cite[Proposition 7.2]{BM}.

\begin{theorem}\label{T1} We have:

\begin{enumerate}

\item[(i)] If $m_{*,q}\leq 1$, then
$$\lim_{n\to\infty} m_{\boldsymbol{\nu},q}^{-n} Z(n)=0  \qquad \text{in $\P_{ \ks }$-probability.}$$

\item[(ii)] If $m_{*,q}>1$,
then
$$\lim_{n\to\infty} m_{\boldsymbol{\nu},q}^{-n} Z(n) = \lim_{n\to\infty} m_{*,q}^{-n} Z_*(n)\coloneqq W_*\qquad \text{in }L^1(\P_{ \ks }).$$
Moreover  the events $\{W_*=0\}$ and $\{\exists n\geq 1: Z_*(n)=0\}$ coincide $\P_{ \ks }$-a.s.
\end{enumerate}
\end{theorem}

\begin{remark}
Comparing Theorem~\ref{T1}(i) with \eqref{E:asympmoyen} shows that the main contribution to the mean population size $\E_{ \ks } (Z(n))$ when $m_{*,q}\leq 1$ is actually due to exceptional events for which  the population is much larger than expected, i.e. $Z(n) \gg m_{\boldsymbol{\nu},q}^{n}$. When $m_{*,q}>1$, this is no longer the case. However, we expect the survival event of $Z$ to be strictly larger than $\{W_* = 0\}$.
\end{remark}

\begin{remark}
Let $(\nu(k))_{k \geq 0}$ be a probability distribution on $\Z_+$ with unbounded support and $q > 0$. If we define a reinforced Galton-Watson process $Z$ with these parameters, then for all $k \in \Z_+$, we can define $Z^{(k)}$  counting individuals in the reinforced process such that none of their ancestor had $k$ ancestors or more. Setting
\[
  \forall 1 \leq j \leq k, \nu^{(k)}(j) = \nu(j) \quad \text{ and } \quad \nu^{(k)}(0) = 1 - \nu([1,k]),
\]
we observe that $Z^{(k)}$ is a $(\boldsymbol{\nu}^{(k)},q)$ reinforced Galton-Watson process, and that $m_{\boldsymbol{\nu}^{{(k)}},q} \geq m_{*,q} \geq kq$. As a result, we deduce from Theorem \ref{T1} that for all $A > 0$, we have $\liminf_{n \to \infty} Z_n/A^n  > 0$ with positive probability, and $(Z_n)$ grows super-exponentially fast. This justifies our choice of only considering compactly supported measures $\boldsymbol{\nu}$.
\end{remark}

Our second main result concerns survival probabilities. For usual Galton-Watson processes,  $Z_\GW$ survives with strictly positive probability if and only
it is super-critical, that is $m_\GW> 1$. In the reinforced setting, Theorem~\ref{T1}(i) entails that $Z$ becomes eventually extinct a.s. whenever $m_{\boldsymbol{\nu},q}\leq 1$; however
we conjectured in \cite{BM} that there should exist reproduction laws $\boldsymbol{\nu}$ and reinforcement parameters $q$ such that $m_{\boldsymbol{\nu},q}> 1$ and nonetheless,  the reinforced Galton-Watson process becomes extinct eventually almost surely.

In the converse direction, observe first that if $q \ks  \geq 1$, then  the  true Galton-Watson process $Z_*$ is supercritical, and the reinforced Galton-Watson process $Z$ obviously survives with strictly positive probability. In particular, observe that if $\boldsymbol{\nu}$ had unbounded support, then the reinforced Galton-Watson process would always survive with positive probability. Assuming now that $q \ks < 1$, we point at the following  sufficient condition for survival of $Z$.

\begin{theorem}\label{T2} Suppose that $q \ks  <  1$, and that
\begin{equation}  \label{eqn:supercritical}
  \sum_{j=1}^{ \ks } \frac{(1-q)j\nu(j)}{1-qj} > 1.
  \end{equation}
Then the reinforced Galton-Watson process survives forever with strictly positive probability, that is for any $\ell \geq 1$ with $\nu(\ell)>0$, we have

$$
  \P_{\ell} (Z(n)\geq 1 \text{ for all }n\geq 1) > 0.
$$
\end{theorem}
\begin{remark} \begin{itemize}
\item[(i)]  The condition \eqref{eqn:supercritical} obviously holds whenever there exists some $j\leq \ks$ such that
$$ \frac{(1-q)j\nu(j)}{1-qj} > 1 .$$
This inequality can be written as $j(q+(1-q)\nu(j)) >1$, and we observe that the left-hand side is the mean reproduction number of the true Galton-Watson sub-process that  results from the reinforced one by suppressing every progeny (together with its descent) with  size different from $j$. That is, the  true Galton-Watson sub-process is supercritical, hence it survives with strictly positive probability. \textit{A fortiori} the same holds for the reinforced Galton-Watson process.

\item[(ii)] In the same vein,  \eqref{eqn:supercritical} also clearly holds when  the usual Galton-Watson process $Z_{\GW}$ is supercritical, viz. $m_{\GW}>1$.
In the converse direction, it is easy to construct reproductions laws $\boldsymbol{\nu}$ with
$$ m_\GW  < 1 < \sum_{j=1}^{ \ks } \frac{(1-q)j\nu(j)}{1-qj}.$$
This provides further examples of reinforced process $Z$ that may survive in situations where usual Galton-Watson process $Z_\GW$ becomes eventually extinct a.s.

\item[(iii)] The sufficient condition \eqref{eqn:supercritical} corresponds exactly to the necessary and sufficient condition for the survival of the branching process $\tilde{Z}$ in which parents beget children according to the law $\boldsymbol{\nu}$ with probability $1-q$, or beget exactly as many children as their own parent with probability $q$. The process $\tilde{Z}$ being a multitype Galton-Watson process, this claim  is easily checked by studying the Perron-Frobenius eigenvalue of the mean reproduction matrix $(q j \delta_{i,j} +(1-q)j\nu(j))_{i,j}$, where $\delta_{\cdot,\cdot}$ is the Kronecker symbol, see \cite[Chapter 4]{AN}. \end{itemize}
\end{remark}

Theorem~\ref{T1} implies that $m_{\boldsymbol{\nu},q} > 1$ is a necessary condition for survival. Theorem~\ref{T2} gives a different sufficient condition for survival, namely \eqref{eqn:supercritical}.  As a consequence, \eqref{eqn:supercritical} entails $m_{\boldsymbol{\nu},q} > 1$; however, we have not been able to check this fact by  purely analytic considerations. Despite these two conditions being quite close to one another, the two are not identical as is being illustrated in Figure~\ref{fig:phaseDiagram}.
\begin{figure}[h]
\centering
\begin{tikzpicture}[xscale=0.8,yscale=1.4]
  \draw[->] (-0.1,0) -- (5.9,0) node[below] {$p$};
  \draw[->] (0,-0.1) -- (0,3) node[left] {$q$};
  \draw (-0.1,2.5) node[left] {$0.25$} -- (0.1,2.5);
  \draw (5,-0.1) node[below] {$0.1$} -- (5,0.1);
  \draw (4,2) node {survival};
  \draw (1.5,0.6) node {extinction};
  \fill[color=black!20] (0.0,2.495) -- (0.02,2.49) -- (0.04,2.4825) -- (0.06,2.4775) -- (0.08,2.4725) -- (0.1,2.465) -- (0.12,2.46) -- (0.14,2.4525) -- (0.16,2.4475) -- (0.18,2.44) -- (0.2,2.435) -- (0.22,2.43) -- (0.24,2.4225) -- (0.26,2.4175) -- (0.28,2.41) -- (0.3,2.405) -- (0.32,2.3975) -- (0.34,2.3925) -- (0.36,2.385) -- (0.38,2.38) -- (0.4,2.3725) -- (0.42,2.3675) -- (0.44,2.36) -- (0.46,2.355) -- (0.48,2.3475) -- (0.5,2.3425) -- (0.52,2.335) -- (0.54,2.33) -- (0.56,2.3225) -- (0.58,2.3175) -- (0.6,2.31) -- (0.62,2.3025) -- (0.64,2.2975) -- (0.66,2.29) -- (0.68,2.285) -- (0.7,2.2775) -- (0.72,2.27) -- (0.74,2.265) -- (0.76,2.2575) -- (0.78,2.2525) -- (0.8,2.245) -- (0.82,2.2375) -- (0.84,2.2325) -- (0.86,2.225) -- (0.88,2.2175) -- (0.9,2.2125) -- (0.92,2.205) -- (0.94,2.1975) -- (0.96,2.1925) -- (0.98,2.185) -- (1.0,2.1775) -- (1.02,2.1725) -- (1.04,2.165) -- (1.06,2.1575) -- (1.08,2.15) -- (1.1,2.145) -- (1.12,2.1375) -- (1.14,2.13) -- (1.16,2.1225) -- (1.18,2.1175) -- (1.2,2.11) -- (1.22,2.1025) -- (1.24,2.095) -- (1.26,2.0875) -- (1.28,2.0825) -- (1.3,2.075) -- (1.32,2.0675) -- (1.34,2.06) -- (1.36,2.0525) -- (1.38,2.045) -- (1.4,2.04) -- (1.42,2.0325) -- (1.44,2.025) -- (1.46,2.0175) -- (1.48,2.01) -- (1.5,2.0025) -- (1.52,1.995) -- (1.54,1.9875) -- (1.56,1.98) -- (1.58,1.9725) -- (1.6,1.965) -- (1.62,1.9575) -- (1.64,1.95) -- (1.66,1.9425) -- (1.68,1.935) -- (1.7,1.9275) -- (1.72,1.92) -- (1.74,1.9125) -- (1.76,1.905) -- (1.78,1.8975) -- (1.8,1.89) -- (1.82,1.8825) -- (1.84,1.875) -- (1.86,1.8675) -- (1.88,1.86) -- (1.9,1.8525) -- (1.92,1.845) -- (1.94,1.835) -- (1.96,1.8275) -- (1.98,1.82) -- (2.0,1.8125) -- (2.02,1.805) -- (2.04,1.7975) -- (2.06,1.7875) -- (2.08,1.78) -- (2.1,1.7725) -- (2.12,1.765) -- (2.14,1.7575) -- (2.16,1.7475) -- (2.18,1.74) -- (2.2,1.7325) -- (2.22,1.7225) -- (2.24,1.715) -- (2.26,1.7075) -- (2.28,1.6975) -- (2.3,1.69) -- (2.32,1.6825) -- (2.34,1.6725) -- (2.36,1.665) -- (2.38,1.6575) -- (2.4,1.6475) -- (2.42,1.64) -- (2.44,1.63) -- (2.46,1.6225) -- (2.48,1.6125) -- (2.5,1.605) -- (2.52,1.595) -- (2.54,1.5875) -- (2.56,1.5775) -- (2.58,1.57) -- (2.6,1.56) -- (2.62,1.5525) -- (2.64,1.5425) -- (2.66,1.535) -- (2.68,1.525) -- (2.7,1.515) -- (2.72,1.5075) -- (2.74,1.4975) -- (2.76,1.4875) -- (2.78,1.48) -- (2.8,1.47) -- (2.82,1.46) -- (2.84,1.45) -- (2.86,1.4425) -- (2.88,1.4325) -- (2.9,1.4225) -- (2.92,1.4125) -- (2.94,1.4025) -- (2.96,1.3925) -- (2.98,1.385) -- (3.0,1.375) -- (3.02,1.365) -- (3.04,1.355) -- (3.06,1.345) -- (3.08,1.335) -- (3.1,1.325) -- (3.12,1.315) -- (3.14,1.305) -- (3.16,1.295) -- (3.18,1.285) -- (3.2,1.275) -- (3.22,1.265) -- (3.24,1.2525) -- (3.26,1.2425) -- (3.28,1.2325) -- (3.3,1.2225) -- (3.32,1.2125) -- (3.34,1.2) -- (3.36,1.19) -- (3.38,1.18) -- (3.4,1.1675) -- (3.42,1.1575) -- (3.44,1.1475) -- (3.46,1.135) -- (3.48,1.125) -- (3.5,1.1125) -- (3.52,1.1025) -- (3.54,1.09) -- (3.56,1.08) -- (3.58,1.0675) -- (3.6,1.0575) -- (3.62,1.045) -- (3.64,1.0325) -- (3.66,1.0225) -- (3.68,1.01) -- (3.7,0.9975) --  (3.72,0.985) --  (3.74,0.9725) --  (3.76,0.9625) --  (3.78,0.95) --  (3.8,0.9375) --  (3.82,0.925) --  (3.84,0.9125) --  (3.86,0.9) --  (3.88,0.8875) --  (3.9,0.875) --  (3.92,0.86) --  (3.94,0.8475) --  (3.96,0.835) --  (3.98,0.8225) --  (4.0,0.81) --  (4.02,0.795) --  (4.04,0.7825) --  (4.06,0.7675) --  (4.08,0.755) --  (4.1,0.74) --  (4.12,0.7275) --  (4.14,0.7125) --  (4.16,0.7) --  (4.18,0.685) --  (4.2,0.67) --  (4.22,0.6575) --  (4.24,0.6425) --  (4.26,0.6275) --  (4.28,0.6125) --  (4.3,0.5975) --  (4.32,0.5825) --  (4.34,0.5675) --  (4.36,0.5525) --  (4.38,0.5375) --  (4.4,0.5225) --  (4.42,0.505) --  (4.44,0.49) --  (4.46,0.475) --  (4.48,0.4575) --  (4.5,0.4425) --  (4.52,0.425) --  (4.54,0.41) --  (4.56,0.3925) --  (4.58,0.375) --  (4.6,0.36) --  (4.62,0.3425) --  (4.64,0.325) --  (4.66,0.3075) -- (4.66,0.325) --  (4.64,0.3425) --  (4.62,0.36) --  (4.6,0.3775) --  (4.58,0.395) --  (4.56,0.4125) --  (4.54,0.43) --  (4.52,0.4475) --  (4.5,0.465) --  (4.48,0.4825) --  (4.46,0.5) --  (4.44,0.5175) --  (4.42,0.535) --  (4.4,0.55) --  (4.38,0.5675) --  (4.36,0.585) --  (4.34,0.6) --  (4.32,0.6175) --  (4.3,0.6325) --  (4.28,0.65) --  (4.26,0.665) --  (4.24,0.6825) --  (4.22,0.6975) --  (4.2,0.7125) --  (4.18,0.73) --  (4.16,0.745) --  (4.14,0.76) --  (4.12,0.775) --  (4.1,0.7925) --  (4.08,0.8075) --  (4.06,0.8225) --  (4.04,0.8375) --  (4.02,0.8525) --  (4.0,0.8675) --  (3.98,0.8825) --  (3.96,0.8975) --  (3.94,0.91) --  (3.92,0.925) --  (3.9,0.94) --  (3.88,0.955) --  (3.86,0.97) --  (3.84,0.9825) --  (3.82,0.9975) --  (3.8,1.0125) --  (3.78,1.025) --  (3.76,1.04) --  (3.74,1.0525) --  (3.72,1.0675) --  (3.7,1.08) --  (3.68,1.095) --  (3.66,1.1075) --  (3.64,1.1225) --  (3.62,1.135) --  (3.6,1.1475) --  (3.58,1.16) --  (3.56,1.175) --  (3.54,1.1875) --  (3.52,1.2) --  (3.5,1.2125) --  (3.48,1.225) --  (3.46,1.24) --  (3.44,1.2525) --  (3.42,1.265) --  (3.4,1.2775) --  (3.38,1.29) --  (3.36,1.3025) --  (3.34,1.3125) --  (3.32,1.325) --  (3.3,1.3375) --  (3.28,1.35) --  (3.26,1.3625) --  (3.24,1.375) --  (3.22,1.385) --  (3.2,1.3975) --  (3.18,1.41) --  (3.16,1.42) --  (3.14,1.4325) --  (3.12,1.445) --  (3.1,1.455) --  (3.08,1.4675) --  (3.06,1.4775) --  (3.04,1.49) --  (3.02,1.5) --  (3.0,1.5125) --  (2.98,1.5225) --  (2.96,1.5325) --  (2.94,1.545) --  (2.92,1.555) --  (2.9,1.565) --  (2.88,1.5775) --  (2.86,1.5875) --  (2.84,1.5975) --  (2.82,1.6075) --  (2.8,1.6175) --  (2.78,1.6275) --  (2.76,1.64) --  (2.74,1.65) --  (2.72,1.66) --  (2.7,1.67) --  (2.68,1.68) --  (2.66,1.69) --  (2.64,1.7) --  (2.62,1.71) --  (2.6,1.7175) --  (2.58,1.7275) --  (2.56,1.7375) --  (2.54,1.7475) --  (2.52,1.7575) --  (2.5,1.765) --  (2.48,1.775) --  (2.46,1.785) --  (2.44,1.795) --  (2.42,1.8025) --  (2.4,1.8125) --  (2.38,1.8225) --  (2.36,1.83) --  (2.34,1.84) --  (2.32,1.8475) --  (2.3,1.8575) --  (2.28,1.865) --  (2.26,1.875) --  (2.24,1.8825) --  (2.22,1.8925) --  (2.2,1.9) --  (2.18,1.9075) --  (2.16,1.9175) --  (2.14,1.925) --  (2.12,1.9325) --  (2.1,1.94) --  (2.08,1.95) --  (2.06,1.9575) --  (2.04,1.965) --  (2.02,1.9725) --  (2.0,1.98) --  (1.98,1.99) --  (1.96,1.9975) --  (1.94,2.005) --  (1.92,2.0125) --  (1.9,2.02) --  (1.88,2.0275) --  (1.86,2.035) --  (1.84,2.0425) --  (1.82,2.05) --  (1.8,2.0575) --  (1.78,2.0625) --  (1.76,2.07) --  (1.74,2.0775) --  (1.72,2.085) --  (1.7,2.0925) --  (1.68,2.1) --  (1.66,2.105) --  (1.64,2.1125) --  (1.62,2.12) --  (1.6,2.125) --  (1.58,2.1325) --  (1.56,2.14) --  (1.54,2.145) --  (1.52,2.1525) --  (1.5,2.16) --  (1.48,2.165) --  (1.46,2.1725) --  (1.44,2.1775) --  (1.42,2.185) --  (1.4,2.19) --  (1.38,2.1975) --  (1.36,2.2025) --  (1.34,2.2075) --  (1.32,2.215) --  (1.3,2.22) --  (1.28,2.225) --  (1.26,2.2325) --  (1.24,2.2375) --  (1.22,2.2425) --  (1.2,2.25) --  (1.18,2.255) --  (1.16,2.26) --  (1.14,2.265) --  (1.12,2.27) --  (1.1,2.275) --  (1.08,2.2825) --  (1.06,2.2875) --  (1.04,2.2925) --  (1.02,2.2975) --  (1.0,2.3025) --  (0.98,2.3075) --  (0.96,2.3125) --  (0.94,2.3175) --  (0.92,2.3225) --  (0.9,2.3275) --  (0.88,2.3325) --  (0.86,2.3375) --  (0.84,2.3425) --  (0.82,2.345) --  (0.8,2.35) --  (0.78,2.355) --  (0.76,2.36) --  (0.74,2.365) --  (0.72,2.37) --  (0.7,2.3725) --  (0.68,2.3775) --  (0.66,2.3825) --  (0.64,2.385) --  (0.62,2.39) --  (0.6,2.395) --  (0.58,2.3975) --  (0.56,2.4025) --  (0.54,2.4075) --  (0.52,2.41) --  (0.5,2.415) --  (0.48,2.42) --  (0.46,2.4225) --  (0.44,2.4275) --  (0.42,2.43) --  (0.4,2.435) --  (0.38,2.4375) --  (0.36,2.4425) --  (0.34,2.445) --  (0.32,2.4475) --  (0.3,2.4525) --  (0.28,2.455) --  (0.26,2.46) --  (0.24,2.4625) --  (0.22,2.465) --  (0.2,2.47) --  (0.18,2.4725) --  (0.16,2.475) --  (0.14,2.48) --  (0.12,2.4825) --  (0.1,2.485) --  (0.08,2.49) --  (0.06,2.4925) --  (0.04,2.495) --  (0.02,2.4975) -- cycle;
  \draw [color=blue,thick] (0.0,2.495) --  (0.02,2.49) --  (0.04,2.4825) --  (0.06,2.4775) --  (0.08,2.4725) --  (0.1,2.465) --  (0.12,2.46) --  (0.14,2.4525) --  (0.16,2.4475) --  (0.18,2.44) --  (0.2,2.435) --  (0.22,2.43) --  (0.24,2.4225) --  (0.26,2.4175) --  (0.28,2.41) --  (0.3,2.405) --  (0.32,2.3975) --  (0.34,2.3925) --  (0.36,2.385) --  (0.38,2.38) --  (0.4,2.3725) --  (0.42,2.3675) --  (0.44,2.36) --  (0.46,2.355) --  (0.48,2.3475) --  (0.5,2.3425) --  (0.52,2.335) --  (0.54,2.33) --  (0.56,2.3225) --  (0.58,2.3175) --  (0.6,2.31) --  (0.62,2.3025) --  (0.64,2.2975) --  (0.66,2.29) --  (0.68,2.285) --  (0.7,2.2775) --  (0.72,2.27) --  (0.74,2.265) --  (0.76,2.2575) --  (0.78,2.2525) --  (0.8,2.245) --  (0.82,2.2375) --  (0.84,2.2325) --  (0.86,2.225) --  (0.88,2.2175) --  (0.9,2.2125) --  (0.92,2.205) --  (0.94,2.1975) --  (0.96,2.1925) --  (0.98,2.185) --  (1.0,2.1775) --  (1.02,2.1725) --  (1.04,2.165) --  (1.06,2.1575) --  (1.08,2.15) --  (1.1,2.145) --  (1.12,2.1375) --  (1.14,2.13) --  (1.16,2.1225) --  (1.18,2.1175) --  (1.2,2.11) --  (1.22,2.1025) --  (1.24,2.095) --  (1.26,2.0875) --  (1.28,2.0825) --  (1.3,2.075) --  (1.32,2.0675) --  (1.34,2.06) --  (1.36,2.0525) --  (1.38,2.045) --  (1.4,2.04) --  (1.42,2.0325) --  (1.44,2.025) --  (1.46,2.0175) --  (1.48,2.01) --  (1.5,2.0025) --  (1.52,1.995) --  (1.54,1.9875) --  (1.56,1.98) --  (1.58,1.9725) --  (1.6,1.965) --  (1.62,1.9575) --  (1.64,1.95) --  (1.66,1.9425) --  (1.68,1.935) --  (1.7,1.9275) --  (1.72,1.92) --  (1.74,1.9125) --  (1.76,1.905) --  (1.78,1.8975) --  (1.8,1.89) --  (1.82,1.8825) --  (1.84,1.875) --  (1.86,1.8675) --  (1.88,1.86) --  (1.9,1.8525) --  (1.92,1.845) --  (1.94,1.835) --  (1.96,1.8275) --  (1.98,1.82) --  (2.0,1.8125) --  (2.02,1.805) --  (2.04,1.7975) --  (2.06,1.7875) --  (2.08,1.78) --  (2.1,1.7725) --  (2.12,1.765) --  (2.14,1.7575) --  (2.16,1.7475) --  (2.18,1.74) --  (2.2,1.7325) --  (2.22,1.7225) --  (2.24,1.715) --  (2.26,1.7075) --  (2.28,1.6975) --  (2.3,1.69) --  (2.32,1.6825) --  (2.34,1.6725) --  (2.36,1.665) --  (2.38,1.6575) --  (2.4,1.6475) --  (2.42,1.64) --  (2.44,1.63) --  (2.46,1.6225) --  (2.48,1.6125) --  (2.5,1.605) --  (2.52,1.595) --  (2.54,1.5875) --  (2.56,1.5775) --  (2.58,1.57) --  (2.6,1.56) --  (2.62,1.5525) --  (2.64,1.5425) --  (2.66,1.535) --  (2.68,1.525) --  (2.7,1.515) --  (2.72,1.5075) --  (2.74,1.4975) --  (2.76,1.4875) --  (2.78,1.48) --  (2.8,1.47) --  (2.82,1.46) --  (2.84,1.45) --  (2.86,1.4425) --  (2.88,1.4325) --  (2.9,1.4225) --  (2.92,1.4125) --  (2.94,1.4025) --  (2.96,1.3925) --  (2.98,1.385) --  (3.0,1.375) --  (3.02,1.365) --  (3.04,1.355) --  (3.06,1.345) --  (3.08,1.335) --  (3.1,1.325) --  (3.12,1.315) --  (3.14,1.305) --  (3.16,1.295) --  (3.18,1.285) --  (3.2,1.275) --  (3.22,1.265) --  (3.24,1.2525) --  (3.26,1.2425) --  (3.28,1.2325) --  (3.3,1.2225) --  (3.32,1.2125) --  (3.34,1.2) --  (3.36,1.19) --  (3.38,1.18) --  (3.4,1.1675) --  (3.42,1.1575) --  (3.44,1.1475) --  (3.46,1.135) --  (3.48,1.125) --  (3.5,1.1125) --  (3.52,1.1025) --  (3.54,1.09) --  (3.56,1.08) --  (3.58,1.0675) --  (3.6,1.0575) --  (3.62,1.045) --  (3.64,1.0325) --  (3.66,1.0225) --  (3.68,1.01) --  (3.7,0.9975) --  (3.72,0.985) --  (3.74,0.9725) --  (3.76,0.9625) --  (3.78,0.95) --  (3.8,0.9375) --  (3.82,0.925) --  (3.84,0.9125) --  (3.86,0.9) --  (3.88,0.8875) --  (3.9,0.875) --  (3.92,0.86) --  (3.94,0.8475) --  (3.96,0.835) --  (3.98,0.8225) --  (4.0,0.81) --  (4.02,0.795) --  (4.04,0.7825) --  (4.06,0.7675) --  (4.08,0.755) --  (4.1,0.74) --  (4.12,0.7275) --  (4.14,0.7125) --  (4.16,0.7) --  (4.18,0.685) --  (4.2,0.67) --  (4.22,0.6575) --  (4.24,0.6425) --  (4.26,0.6275) --  (4.28,0.6125) --  (4.3,0.5975) --  (4.32,0.5825) --  (4.34,0.5675) --  (4.36,0.5525) --  (4.38,0.5375) --  (4.4,0.5225) --  (4.42,0.505) --  (4.44,0.49) --  (4.46,0.475) --  (4.48,0.4575) --  (4.5,0.4425) --  (4.52,0.425) --  (4.54,0.41) --  (4.56,0.3925) --  (4.58,0.375) --  (4.6,0.36) --  (4.62,0.3425) --  (4.64,0.325) --  (4.66,0.3075);
    \draw [color = orange, thick] (0.0,2.4975) --  (0.02,2.495) --  (0.04,2.4925) --  (0.06,2.49) --  (0.08,2.485) --  (0.1,2.4825) --  (0.12,2.48) --  (0.14,2.475) --  (0.16,2.4725) --  (0.18,2.47) --  (0.2,2.465) --  (0.22,2.4625) --  (0.24,2.46) --  (0.26,2.455) --  (0.28,2.4525) --  (0.3,2.4475) --  (0.32,2.445) --  (0.34,2.4425) --  (0.36,2.4375) --  (0.38,2.435) --  (0.4,2.43) --  (0.42,2.4275) --  (0.44,2.4225) --  (0.46,2.42) --  (0.48,2.415) --  (0.5,2.41) --  (0.52,2.4075) --  (0.54,2.4025) --  (0.56,2.3975) --  (0.58,2.395) --  (0.6,2.39) --  (0.62,2.385) --  (0.64,2.3825) --  (0.66,2.3775) --  (0.68,2.3725) --  (0.7,2.37) --  (0.72,2.365) --  (0.74,2.36) --  (0.76,2.355) --  (0.78,2.35) --  (0.8,2.345) --  (0.82,2.3425) --  (0.84,2.3375) --  (0.86,2.3325) --  (0.88,2.3275) --  (0.9,2.3225) --  (0.92,2.3175) --  (0.94,2.3125) --  (0.96,2.3075) --  (0.98,2.3025) --  (1.0,2.2975) --  (1.02,2.2925) --  (1.04,2.2875) --  (1.06,2.2825) --  (1.08,2.275) --  (1.1,2.27) --  (1.12,2.265) --  (1.14,2.26) --  (1.16,2.255) --  (1.18,2.25) --  (1.2,2.2425) --  (1.22,2.2375) --  (1.24,2.2325) --  (1.26,2.225) --  (1.28,2.22) --  (1.3,2.215) --  (1.32,2.2075) --  (1.34,2.2025) --  (1.36,2.1975) --  (1.38,2.19) --  (1.4,2.185) --  (1.42,2.1775) --  (1.44,2.1725) --  (1.46,2.165) --  (1.48,2.16) --  (1.5,2.1525) --  (1.52,2.145) --  (1.54,2.14) --  (1.56,2.1325) --  (1.58,2.125) --  (1.6,2.12) --  (1.62,2.1125) --  (1.64,2.105) --  (1.66,2.1) --  (1.68,2.0925) --  (1.7,2.085) --  (1.72,2.0775) --  (1.74,2.07) --  (1.76,2.0625) --  (1.78,2.0575) --  (1.8,2.05) --  (1.82,2.0425) --  (1.84,2.035) --  (1.86,2.0275) --  (1.88,2.02) --  (1.9,2.0125) --  (1.92,2.005) --  (1.94,1.9975) --  (1.96,1.99) --  (1.98,1.98) --  (2.0,1.9725) --  (2.02,1.965) --  (2.04,1.9575) --  (2.06,1.95) --  (2.08,1.94) --  (2.1,1.9325) --  (2.12,1.925) --  (2.14,1.9175) --  (2.16,1.9075) --  (2.18,1.9) --  (2.2,1.8925) --  (2.22,1.8825) --  (2.24,1.875) --  (2.26,1.865) --  (2.28,1.8575) --  (2.3,1.8475) --  (2.32,1.84) --  (2.34,1.83) --  (2.36,1.8225) --  (2.38,1.8125) --  (2.4,1.8025) --  (2.42,1.795) --  (2.44,1.785) --  (2.46,1.775) --  (2.48,1.765) --  (2.5,1.7575) --  (2.52,1.7475) --  (2.54,1.7375) --  (2.56,1.7275) --  (2.58,1.7175) --  (2.6,1.71) --  (2.62,1.7) --  (2.64,1.69) --  (2.66,1.68) --  (2.68,1.67) --  (2.7,1.66) --  (2.72,1.65) --  (2.74,1.64) --  (2.76,1.6275) --  (2.78,1.6175) --  (2.8,1.6075) --  (2.82,1.5975) --  (2.84,1.5875) --  (2.86,1.5775) --  (2.88,1.565) --  (2.9,1.555) --  (2.92,1.545) --  (2.94,1.5325) --  (2.96,1.5225) --  (2.98,1.5125) --  (3.0,1.5) --  (3.02,1.49) --  (3.04,1.4775) --  (3.06,1.4675) --  (3.08,1.455) --  (3.1,1.445) --  (3.12,1.4325) --  (3.14,1.42) --  (3.16,1.41) --  (3.18,1.3975) --  (3.2,1.385) --  (3.22,1.375) --  (3.24,1.3625) --  (3.26,1.35) --  (3.28,1.3375) --  (3.3,1.325) --  (3.32,1.3125) --  (3.34,1.3025) --  (3.36,1.29) --  (3.38,1.2775) --  (3.4,1.265) --  (3.42,1.2525) --  (3.44,1.24) --  (3.46,1.225) --  (3.48,1.2125) --  (3.5,1.2) --  (3.52,1.1875) --  (3.54,1.175) --  (3.56,1.16) --  (3.58,1.1475) --  (3.6,1.135) --  (3.62,1.1225) --  (3.64,1.1075) --  (3.66,1.095) --  (3.68,1.08) --  (3.7,1.0675) --  (3.72,1.0525) --  (3.74,1.04) --  (3.76,1.025) --  (3.78,1.0125) --  (3.8,0.9975) --  (3.82,0.9825) --  (3.84,0.97) --  (3.86,0.955) --  (3.88,0.94) --  (3.9,0.925) --  (3.92,0.91) --  (3.94,0.8975) --  (3.96,0.8825) --  (3.98,0.8675) --  (4.0,0.8525) --  (4.02,0.8375) --  (4.04,0.8225) --  (4.06,0.8075) --  (4.08,0.7925) --  (4.1,0.775) --  (4.12,0.76) --  (4.14,0.745) --  (4.16,0.73) --  (4.18,0.7125) --  (4.2,0.6975) --  (4.22,0.6825) --  (4.24,0.665) --  (4.26,0.65) --  (4.28,0.6325) --  (4.3,0.6175) --  (4.32,0.6) --  (4.34,0.585) --  (4.36,0.5675) --  (4.38,0.55) --  (4.4,0.535) --  (4.42,0.5175) --  (4.44,0.5) --  (4.46,0.4825) --  (4.48,0.465) --  (4.5,0.4475) --  (4.52,0.43) --  (4.54,0.4125) --  (4.56,0.395) --  (4.58,0.3775) --  (4.6,0.36) --  (4.62,0.3425) --  (4.64,0.325) --  (4.66,0.3075) --  (4.68,0.2875) --  (4.7,0.27) --  (4.72,0.2525) --  (4.74,0.2325) --  (4.76,0.215) --  (4.78,0.195) --  (4.8,0.1775) --  (4.82,0.1575) --  (4.84,0.1375) --  (4.86,0.12) --  (4.88,0.1) --  (4.9,0.08) --  (4.92,0.06) --  (4.94,0.04) --  (4.96,0.02) --  (4.98,0.0025) --  (5.0,0.0025);
\end{tikzpicture}
\caption{Phase diagram of a reinforced Galton-Watson process with parameter $(\boldsymbol{\nu}_p,q)$ with $\boldsymbol{\nu}_p = (1 - 4p)\delta_0 + p(\delta_1+\delta_2+\delta_3+\delta_4)$, for $q \in [0,0.25]$ and $p \in [0,0.1]$. The blue line corresponds to the set of parameters such that $m_{\boldsymbol{\nu}_p,q} = 1$, the orange one to parameters such that $\sum (1-q) j \nu_p(j)/(1-qj) = 1$.
 The grey domain corresponds to the set of parameters for which $m_{\boldsymbol{\nu}_p,q} > 1$ but \eqref{eqn:supercritical} does not hold. Note that $m_{\GW}>1$ for $p>0.1$ and that $q\ks=4q\geq 1$ for $q\geq 0.25$, which explains the boundary points of the curves.}
\label{fig:phaseDiagram}
\end{figure}
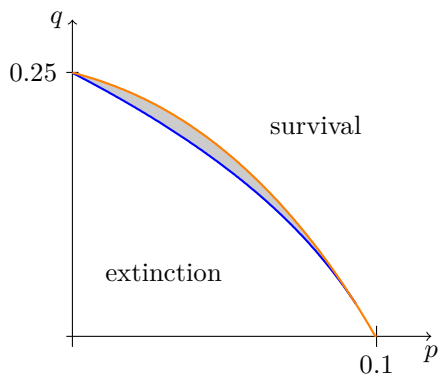

The rest of this work is organized as follows. Section~\ref{sec:const} dwells on the observation that reinforced Galton-Watson processes can be viewed as multi-type branching processes, where the type of an individual corresponds to the counting measure of the number of children of its forebears. This enables us to recover the branching property, of course at the cost of working with a rather large space of types. In Section~\ref{sec:bounds}, we establish some bounds for such multi-type branching processes, which follow rather directly from \eqref{E:asympmoyen} and \eqref{E:asympmoyen0}.
These bounds are then used in Section~\ref{sec:p1} to establish Theorem~\ref{T1}.

Section~\ref{sec:p2} is in turn devoted to the proof of Theorem~\ref{T2}; let us briefly sketch our approach which is fully independent of \cite{BM}. We start in Subsection~\ref{sec:mart}  by constructing a natural nonnegative martingale $M$ for the reinforced Galton-Watson process, which starts from $M_0=1$ and vanishes as soon as $Z$ becomes extinct. This leads us to investigate whether $M$ is uniformly integrable, since in that case, the reinforced Galton-Watson process obviously survives with strictly positive probability.  In this direction, we use $M$ to define a tilted probability law in Subsection~\ref{sec:spine}, and describe the evolution of the reinforced Galton-Watson process under the tilted law  in terms of the so-called  spinal decomposition. Next, in Subsection~\ref{sec:urns}, we interpret the evolution of the types along the spine in terms of a generalized P\'olya urn process. This permits us to determine the asymptotic behavior of  those types. The proof of Theorem~\ref{T2} can then be completed in Subsection~\ref{sec:reinforcedPerpetuity} by applying classical arguments involving the spinal decomposition and Durrett's criterion for the uniform integrability of a nonnegative martingale.

\section{Reinforced Galton-Watson processes as multi-type branching processes}
\label{sec:const}

The reproduction law of  an individual in a reinforced Galton-Watson process depends on its entire ancestral lineage. Different individuals may partly share the same ancestral lineage, and then their respective evolutions are not independent. The Markov and the branching properties are therefore lost in this process, but can nonetheless be recovered by endowing each individual with a type that records the reproduction numbers of its forebears.

To start with, we briefly adapt to our setting the  Ulam--Harris--Neveu framework which enables to encode any population model with non-overlapping generations  started from a single ancestor by its the genealogical tree $\mathcal T$. Since in our model, an individual has never more than $\ks$ children, we use
the space  of finite (possibly empty) sequences in $ \{1, \ldots, \ks\}$,
$$ \mathcal{U} \coloneqq
  \bigcup_{\ell \geq 0} \{1, \ldots, \ks\}^\ell,$$
as  the universal tree that contains $\mathcal T$.

The empty sequence $\varnothing$ is assigned to the ancestor of the population and viewed as the root of the genealogical tree $\mathcal{T}$.
 An individual at the $\ell$-th generation is represented by some ${u}=(u_1, \ldots, u_\ell)\in \mathcal{T}$, and if this individual has $k\geq 1$ children, the latter are represented in turn by the sequences ${u} j\coloneqq(u_1, \ldots, u_\ell,j)$ for $j=1, \ldots, k$. As in our setting, an individual can have $k$ children only for $k\geq 0$  such that $\nu(k)>0$,
 the outer-degree $d({u})$ of any individual ${u} \in \mathcal T$ always belongs to the support $\{k\geq 0: \nu(k)>0\}$ of the reproduction law $\boldsymbol{\nu}$.

We further assign types to individuals, that is to the vertices of the genealogical tree $\mathcal{T}$ as follows. The space of types is  the set  $\mathcal{M}_{\boldsymbol{\nu}}$ of finite integer-valued measures on $\{k\geq 1: \nu(k)>0\}$. Types are associated to vertices of $\mathcal{T}$ using the following (deterministic) algorithm. If an individual $u \in \mathcal{T}$ has type $\mathbf{t}$ and has $k \geq 1$ children at the next generation, then all its children receive the type $\mathbf{t} + \delta_k$. The assignation of types is thus completely determined once the type $\mathbf{t}(\varnothing)$ of the initial ancestor and the genealogical tree $\mathcal{T}$ are known.

The type of the ancestor may be either the null measure $\mathbf{0}$ or a measure of positive mass, $|\mathbf{t}(\varnothing)|=j\geq 1$.
In the latter case, we may imagine that the ancestor $\varnothing$ has also forebears in the past, at generations $-1, \ldots, -j$,
and then $\mathbf{t}(\varnothing)$ should be thought of as to the empirical measure of the numbers of children of the forebears of the ancestor.
In all cases,  the type $\mathbf{t}({u})$ of an individual ${u} \in \mathcal T$ corresponds to the empirical counting measure of the number of children of its forebears, possibly including forebears at negative generations when $\mathbf{t}(\varnothing)\neq \mathbf{0}$.

This setting enables us to view a reinforced Galton-Watson process as a multitype branching process; see \cite{BiK04} for background. Recall that $q \in (0,1)$ is a reinforcement parameter and $\boldsymbol{\nu}$ a probability vector on $\{0,1, \ldots, \ks \}$. We just need to specify the reproduction law $\boldsymbol{\pi}_{\mathbf t}$ of an individual as a function of its type $\mathbf{t}$.  First, when the type is the null measure (which can only be the case for the ancestor), we agree\footnote{Recall that in the first part of this work, we are essentially concerned with the law $\P_{\ks}$ when the ancestor has $\ks $ children, who all have type $\delta_{\ks}$. If we were working under law $\P_{\ell}$ for $1 \leq \ell \leq \ks$, then we would have imposed that an individual with the null type as exactly $\ell$ children all with type $\delta_\ell$.} that the individual has almost surely $\ks$ children all with type $\delta_\ks$, that is
$$\boldsymbol{\pi}_{\mathbf{0}}=\delta_{\ks}.$$
Next, when the type is non-zero, say $\mathbf{t}$, the probability that this individual begets no child equals
$${\pi}_{\mathbf t}(0)=(1-q)\nu(0),$$ and  the probability that it begets exactly $k\geq 1$ children all of type $\mathbf{t}+\delta_k$ and no further children equals
 $${\pi}_{\mathbf t}(k)=(1-q)\nu(k) + q\frac{{t}(k)}{|\mathbf{t}|},$$
where $|\mathbf{t}|$ stands for the total mass of $\mathbf{t}$ and $t(k)$ for the mass assigned by $\mathbf t$ to the atom $k$. In particular,  $\mathbf{t}/|\mathbf{t}|$ is the empirical distribution of the numbers of children of the forebears of an individual.

For every tree $\mathcal{T}$ and vertex $ {u}\in \mathcal{T}$, we write $\mathcal{T}_{u}$ for the subtree that stems from ${u}$. Specifically a vertex $v$ belongs to $\mathcal{T}_{u}$ if and only if ${u} v$ (the concatenation of the words ${u}$ and $v$) is a vertex of $\mathcal{T}$, and then we agree that the type of $v$ in  $\mathcal{T}_{u}$ is the one of ${u} v$ in $\mathcal{T}$. We also write
$$\mathcal{T}_{|\ell} \coloneqq \{{u}\in \mathcal{T} : |{u}|\leq \ell\}$$
for the restriction of the genealogical tree to the first $j$ generations.

For every type $\mathbf{t}\in \mathcal{M}_{\boldsymbol{\nu}}$, we write $\mathcal{P}_{\mathbf{t}}$ for the law of the genealogical tree $\mathcal T$ of the multitype branching process with the reproductions laws described above conditioned on the ancestor having type $\mathbf{t}$. With this formalism at hand, we can now recover the branching property. The basic Markov-branching property for the multitype branching process can be stated as follows.

\begin{lemma}
\label{L:branch}
Let $T$ be a fixed tree of height $\ell$ (i.e. such that $\sup_{v \in T} |v| = \ell)$ and an arbitrary initial type $\mathbf{t} \in \mathcal{M}_{\boldsymbol{\nu}}$, we denote by $\mathbf{t}({u})$ the type of the vertex ${u} \in T$ by fixing $\mathbf{t}(\varnothing) = \mathbf{t}$. Then, under the conditioned probability measure $\mathcal{P}_{\mathbf{t}}(\cdot | \mathcal{T}_{|\ell} = T)$, the subtrees $(\mathcal{T}_{u}: {u} \in \mathcal{T} \text{ and } |{u}| = \ell)$ are independent and each $\mathcal{T}_u$ has the law $\mathcal{P}_{\mathbf{t}({u})}$.
\end{lemma}

We will also use in the sequel the following consequence of Lemma~\ref{L:branch}. For any  initial type $\mathbf{t}\in \mathcal{M}_{\boldsymbol{\nu}}$, every vertex ${u}\in \mathcal{U}$, and every $d\geq 1$ such that $\nu(d)>0$, under the conditional probability measure $\mathcal{P}_{\mathbf{t}}$ given that ${u}\in \mathcal{T}$, that ${u}$ has outer degree $d({u})=d$ in $\mathcal{T}$ and type $\mathbf{s}$, the $d$ subtrees $\mathcal{T}_{{u}1}, \ldots, \mathcal{T}_{{u}d}$ are independent and each has  the law $\mathcal{P}_{\mathbf{s}+\delta_d}$.

We further point at a useful domination property which should be intuitively obvious.
\begin{lemma} \label{L:dom}
Consider two types $\mathbf{t},\mathbf{t'}\in \mathcal{M}_{\boldsymbol{\nu}}$ such that $|\mathbf{t}|=|\mathbf{t'}|\geq 1$ and
\[\mathbf{t}([\ell,\ks]) \leq \mathbf{t}'([\ell,\ks]) \text{ for all $\ell \in \{1,\ldots,\ks\}$}.\]
In words,  $\mathbf{t}$ and $\mathbf{t'}$ have same total mass and the normalized probability measure $\mathbf{t'}/|\mathbf{t}'|$ stochastically dominates $\mathbf{t}/|\mathbf{t}|$. Then we can construct two random genealogical trees $\mathcal T$ and $\mathcal T'$, the first with the law $\mathcal{P}_{\mathbf{t}}$ and the second with the law $\mathcal{P}_{\mathbf{t'}}$, such that $\mathcal T$ is a subtree of $\mathcal T'$.
\end{lemma}

\begin{proof}
We immediately see by inspection of the reproduction laws of the ancestor as a function of its type, that we can couple the number of children $Z(1)$ of the ancestor under law $\mathcal{P}_{\mathbf{t}}$ with that $Z'(1)$ under $\mathcal{P}_{\mathbf{t}'}$, such that $Z(1)\leq Z'(1)$ by the stochastic domination condition. Then the types $\mathbf{t}+\delta_{Z(1)}$ of the individuals at the first generation under $\mathcal{P}_{\mathbf{t}}$ are also dominated by the types $\mathbf{t'} + \delta_{Z'(1)}$ under $\mathcal{P}_{\mathbf{t}}$. We can then complete the proof by induction using the branching property of Lemma~\ref{L:branch}.
\end{proof}

\section{Some useful bounds}
\label{sec:bounds}
The purpose of this section is to establish some inequalities that will be needed in the proof of Theorem~\ref{T1}. These follow readily from the framework develop in the preceding section and the key estimates \eqref{E:asympmoyen} and \eqref{E:asympmoyen0}. For every $n\geq 0$, the integer-valued measure
$$\mathcal{Z}_n \coloneqq \sum_{|{u}|=n, {u}\in \mathcal{T}} \delta_{\mathbf{t}({u})}, \qquad n\geq 0,$$
that counts the number of individuals at the $n$-th generation as a function of their types, can be viewed as an enriched version of the reinforced Galton-Watson process for which the types of individuals are recorded. In particular, the total mass of $|\mathcal{Z}_n|\coloneqq \mathcal{Z}_n(\mathcal M_{\boldsymbol{\nu}})$ under $\mathcal P_{\mathbf{0}}$ coincides with $Z(n)$ under $\P_{\ks}$.

We start with the following generalization of the bounds \eqref{E:asympmoyen} and \eqref{E:asympmoyen0} to arbitrary fixed initial types. Recall the notation \eqref{E:I} and \eqref{E:m*}.

\begin{lemma} \label{L:L1boundbis} Consider a type $\mathbf{t} \in \mathcal{M}_{\boldsymbol{\nu}}$ with $|\mathbf t | = \ell$.
As $n\to \infty$, we have
$$\mathcal E_{\mathbf{t}}\left( |\mathcal{Z}_n|\right) \sim \left( \frac{m_{\boldsymbol{\nu},q}}{m_{*,q}}\right)^{\ell} m_{\boldsymbol{\nu},q}^{n}$$
when ${t}(\ks) = \ell$ (i.e. $\mathbf t = \ell \delta_\ks$), whereas
$$\mathcal E_{\mathbf{t}}\left( |\mathcal{Z}_n|\right) = o\left( m_{\boldsymbol{\nu},q}^n\right)$$
otherwise.
\end{lemma}
We mention that the approach in \cite{BM} would yield a much sharper estimate in the second case.
However, calculations would be technically rather demanding and we prefer to establish the weaker result using only a much simpler argument, as this suffices for our purpose.

\begin{proof} We consider first the case when type of the ancestor of the reinforced Galton-Watson process is the measure $\delta_j$ and work under $\mathcal P_{\delta_j}$. We may imagine that the ancestor had a parent at generation $-1$ and $j-1$ siblings for some $1\leq j<\ks$ with $\nu(j)>0$. This yields the identity
$$j \mathcal E_{\delta_j}\left( |\mathcal{Z}_n|\right)= \E_j(Z(n+1)),$$
and \eqref{E:asympmoyen0} entails that for any fixed $\ell \geq 0$,
$$\mathcal E_{\delta_j}\left( |\mathcal{Z}_{\ell + n}|\right)= o\left( m_{\boldsymbol{\nu},q}^n\right) \text{ as $n \to \infty$}.$$
The probability under $\mathcal P_{\delta_j}$ that the individual $(1,\ldots, 1)$ at generation $\ell$ is present in the population and has type $\delta_j + \ell\delta_{\ks}$ (i.e. all its forebears at generations $0, \ldots, \ell-1$ had $\ks$ children) is no less than $((1-q)\nu(\ks))^{\ell}>0$. It follows from the branching property that
$$\mathcal E_{\delta_j + \ell \delta_\ks}\left( |\mathcal{Z}_{n}|\right) \leq ((1-q)\nu(\ks))^{-\ell} \mathcal E_{\delta_j}\left( |\mathcal{Z}_{\ell + n}|\right)= o\left( m_{\boldsymbol{\nu},q}^n\right).$$
An application of Lemma~\ref{L:dom} completes the second claim of the statement.

We next work under $\mathcal P_{\mathbf{0}}$; recall from \eqref{E:asympmoyen} that
$$\mathcal E_{\mathbf{0}}\left( |\mathcal{Z}_{n+\ell}|\right)= \E_{\ks}(Z(n+\ell))\sim \frac{m_{\boldsymbol{\nu},q}^{n+\ell}}{q+
(1-q) \nu( \ks )} . $$
We apply the branching property under $\mathcal{P}_\mathbf{0}$ at the $\ell$-th generation. Recall that the number of individuals of type $\ell \delta_\ks$ is
$$ \mathcal Z_{\ell}(\{\ell \delta_\ks\}) = Z_*(\ell),$$
and that $Z_*$ is a usual Galton-Watson process with mean reproduction number $m_{*,q}$ given that the ancestor has $\ks$ children.
In particular, the mean number of these individuals is
\begin{equation} \label{E:mean*}
\E_{\ks}(Z_*(\ell))= \ks m_{*,q}^{\ell-1}.
\end{equation}
There are also $Z(\ell)-Z_*(\ell) \leq k_*^{\ell}$ individuals with types different from $\ell \delta_{\ks}$, and we know from the first part of this proof that as $n\to \infty$, the average  descent of each of them at generation $n+\ell$ is $o\left( m_{\boldsymbol{\nu},q}^n\right)$. Therefore, the branching property yields
$$\mathcal E_{\mathbf{0}}\left( |\mathcal{Z}_{n+\ell}|\right) \sim \E_{\ks}(Z_*(\ell)) \mathcal E_{\ell \delta_{\ks}}\left( |\mathcal{Z}_{n}|\right),$$
and we conclude that
$$ \mathcal E_{\ell \delta_{\ks}}\left( |\mathcal{Z}_{n}|\right) \sim  \frac{ 1}{q+
(1-q) \nu( \ks )} \,  \frac{m_{*,q}}{\ks}\, \left( \frac{m_{\boldsymbol{\nu},q}}{m_{*,q}}\right)^{\ell} m_{\boldsymbol{\nu},q}^{n}= \left( \frac{m_{\boldsymbol{\nu},q}}{m_{*,q}}\right)^{\ell} m_{\boldsymbol{\nu},q}^{n}, $$
using \eqref{E:m*} for the last equality.
\end{proof}

A similar argument also yields  the following uniform bounds.

\begin{lemma} \label{L:L1bound}
There is some finite constant $c_{\boldsymbol{\nu},q}$ depending only on the reinforcement parameter and the reproduction law, such that for any type $\mathbf{t} \in \mathcal{M}_{\boldsymbol{\nu}}$ and any $n\geq0$, one has
$$\mathcal E_{\mathbf{t}}\left( |\mathcal{Z}_n|\right) \leq c_{\boldsymbol{\nu},q} \left(\frac{ m_{\boldsymbol{\nu},q}}{m_{*,q}} \right)^{|\mathbf{t}|} m_{\boldsymbol{\nu},q}^n.$$
\end{lemma}

\begin{proof}
Consider first the case when $\mathbf{t}=\mathbf{0}$, so $|\mathbf{t}|=0$. Recall that
$$\mathcal E_{\mathbf{0}}\left( |\mathcal{Z}_n|\right)= \E_{\ks}(Z(n)),$$
so by \eqref{E:asympmoyen}, we can find some finite constant $c$ such that
$$\mathcal E_{\mathbf{0}}\left( |\mathcal{Z}_n|\right) \leq c m_{\boldsymbol{\nu},q}^n \qquad\text{for all }n\geq 0.$$

Next, as it was already discussed previously, by focussing on individuals at a given  generation $\ell \geq 1$
whose forebears all had $\ks$
children and applying the branching property of Lemma~\ref{L:branch}, we get the inequality
$$\mathcal E_{\mathbf{0}} \left( |\mathcal{Z}_{n+\ell}|\right) \geq \E_{\ks}(Z_*(\ell))
\mathcal E_{\ell\delta_\ks}\left( |\mathcal{Z}_n|\right).$$
The identity \eqref{E:mean*} entails our claim whenever $t(\ks)= |\mathbf{t}|$ with
$c_{\boldsymbol{\nu},q}= c m_{*,q}/\ks$.
Finally, the general case for a type  $\mathbf{t}$ follows from above by an application of Lemma~\ref{L:dom}, since $\mathbf{t}$ is dominated by $|\mathbf{t}|\delta_\ks$.
\end{proof}

When the usual Galton-Watson process $Z_*$ is supercritical, we can also bound the second moment of $|\mathcal{Z}_n|$ by combining these inequalities with a combinatorial argument.
\begin{lemma}
\label{L:new}
If $m_{*,q}>1$, then there exists a constant $C_{\boldsymbol{\nu},q}$ such that for all $n \geq 1$, and $\mathbf{t} \in \mathcal{M}_{\boldsymbol{\nu}}$, we have
\[
  \mathcal{E}_\mathbf{t}\left( |\mathcal{Z}_n|^2 \right) \leq C_{\boldsymbol{\nu},q} \left( \frac{m_{\boldsymbol{\nu},q}}{m_{*,q}} \right)^{2|\mathbf{t}|} m_{\boldsymbol{\nu},q}^{2n}.
\]
\end{lemma}

\begin{proof}
Observe that we can write
\[
  |\mathcal{Z}_n|^2 = |\mathcal{Z}_{n}|\left(|\mathcal{Z}_{n}|-1 \right) + |\mathcal{Z}_n|.
\]
It is therefore enough to bound $\mathcal{E}_{\mathbf{t}}\left(|\mathcal{Z}_{n}|\left(|\mathcal{Z}_{n}|-1 \right) \right)$ (the mean number of couples of distinct individuals alive at generation $n$) to complete the proof.

For $v \in \mathcal{T}$ with $|v| < n$, we denote by $\mathcal{Z}_n^v = \sum_{{u} \in \mathcal{T}^v : |{u}|=n} \delta_{\mathbf{t}(u)}$ the counting measure of the subpopulation at generation $n$ descending from the individual $v$. We can now write
\begin{align*}
  \mathcal{E}_{\mathbf{t}}\left(|\mathcal{Z}_{n}|\left(|\mathcal{Z}_{n}|-1 \right) \right)
  &= \mathcal{E}_{\mathbf{t}}\left( \sum_{|u|=n} \ind{u \in \mathcal{T}} \left(|\mathcal{Z}_{n}|-1 \right) \right)\\
  &= \sum_{|u|=n} \mathcal{E}_{\mathbf{t}}\left(\ind{u \in \mathcal{T}} \sum_{j=0}^{n-1} \sum_{k = 1}^{d(u(j))} \ind{u(j)k \neq u(j+1)} |\mathcal{Z}_n^{u(j)k}| \right),
\end{align*}
where $u(j) = (u_1,\ldots,u_j)$ for $u = (u_1,\ldots,u_n)$. In words, we decompose the $|\mathcal{Z}_n|-1$ individuals alive at generation $n$ barring $u$ according to their most recent common ancestor with $u$. We obtain from the branching property of $\mathcal{Z}$ that
\begin{align*}
  &\phantom{=}\mathcal{E}_{\mathbf{t}}\left(|\mathcal{Z}_{n}|\left(|\mathcal{Z}_{n}|-1 \right) \right)\\
  &= \sum_{|u|=n} \sum_{j=0}^{n-1} \mathcal{E}_{\mathbf{t}}\left( \ind{u \in \mathcal{T}} (d(u(j))-1) \mathcal{E}_{\mathbf{t}(u(j+1))}\left( |\mathcal{Z}_{n-j-1}| \right) \right)\\
  &\leq \sum_{|u|=n} \sum_{j=0}^{n-1} \mathcal{P}_{\mathbf{t}}(u \in \mathcal{T}) \ks c_{\boldsymbol{\nu},q} m_{\boldsymbol{\nu},q}^{n-j-1} \left(\frac{ m_{\boldsymbol{\nu},q}}{m_{*,q}} \right)^{|\mathbf{t}| + j+1},
\end{align*}
where we used that $d(u(j)) \leq \ks$ and $|\mathbf{t}(u(j)k)| = |\mathbf{t}| + j + 1$ $\mathcal{P}_{\mathbf{t}}$-a.s., and we applied Lemma~\ref{L:L1bound}. As $m_{*,q} >1$, we immediately obtain that
\[
  \mathcal{E}_{\mathbf{t}}\left(|\mathcal{Z}_{n}|\left(|\mathcal{Z}_{n}|-1 \right) \right) \leq \frac{\ks c_{\boldsymbol{\nu},q}}{m_{*,q}-1} \left( \frac{m_{\boldsymbol{\nu},q}}{m_{*,q}} \right)^{|\mathbf{t}|} \mathcal{E}_{\mathbf{t}}(|\mathcal{Z}_n|) m_{\boldsymbol{\nu},q}^{n}.
\]
Applying again Lemma~\ref{L:L1bound}, the proof is now complete.
\end{proof}

\section{Proof of Theorem~\ref{T1}}
\label{sec:p1}
 This section is devoted to the proof of Theorem~\ref{T1}. We first show that under the assumption $m_{*,q} \leq 1$, $Z(n) /\E_{ \ks }(Z(n))$ converges to zero in probability. We next prove that if $ m_{*,q}>1$, then $Z(n)/\E_{\ks}(Z(n))$ converges in $L^1$ to a non-degenerate random variable.

$(i)$ We assume here that $m_{*,q}\leq 1$, that is the usual Galton-Watson process $Z_*$ is critical or sub-critical, and therefore becomes eventually extinct $\P_{\ks}$-almost surely.
For any given $\epsilon >0$, we can choose a generation $\ell\geq 1$ sufficiently large so that
$\P_{\ks}(Z_*(\ell)\geq 1)\leq \epsilon^2$. In other words, with $\mathcal{P}_{\mathbf{0}}$-probability at least $1-\epsilon^2$,  all individuals at generation $\ell$ have at least one forebear that had strictly less than $\ks$ children. Plainly, there are at generation $\ell$ at most $\ks^\ell$ individuals, and the set of possible types for these individuals is also bounded by $\ks^\ell$.
The branching property yields
\begin{align*}&\E_{\ks}(m_{\boldsymbol{\nu},q}^{-n-\ell} Z(n+\ell) , Z_*(\ell)=0)\\ &\leq  m_{\boldsymbol{\nu},q}^{-\ell} \ks^\ell
\max\{ m_{\boldsymbol{\nu},q}^{-n} \mathcal E_{\mathbf{t}}\left( |\mathcal{Z}_n|\right), \mathbf t \in \mathcal{M}_{\boldsymbol{\nu}} : |\mathbf{t}| = \ell > {t}(\ks)\},
\end{align*}
 and we now see from Lemma~\ref{L:L1boundbis} that the right-hand side can be bounded from above by $\epsilon^2$ for all sufficiently large $n$.
 An application of the Markov inequality now gives
 $$\lim_{n\to \infty} \P_{\ks}(m_{\boldsymbol{\nu},q}^{-n-\ell} Z(n+\ell)\geq \epsilon) \leq 2\epsilon ,$$
 and  Theorem~\ref{T1}(i) is proven.

 (ii) We now suppose that $m_{*,q}>1$, i.e. that the usual Galton-Watson process $Z_*$ is supercritical. Since its reproduction law has bounded support,  the process
  $$W_*(n)\coloneqq m_{*,q}^{-n}Z_*(n), \qquad n\geq 0$$ is a martingale bounded in $L^2$,  and we write $W_*$ for its terminal value. Recalling that the ancestor has $\ks$ children, we have
  $$ \mathcal{E}_\mathbf{0}(W_*)= \mathcal{E}_\mathbf{0}(W_*(n))= \ks m_{*,q}^{-1} =  \frac{ 1}{q+
 (1-q) \nu( \ks )}.$$
 Note that by \eqref{E:asympmoyen} and \eqref{E:asympmoyen0}, we have
 \[
  \lim_{n\to\infty} \E_{\ks}(m_{\boldsymbol{\nu},q}^{-n} Z(n)) = \frac{1}{q+
  (1-q) \nu( \ks )} = \mathcal{E}_\mathbf{0}(W_*).
 \]

 To prove that $m_{\nu,q}^{-n}Z(n)$ converges to $W_*$ in $L^1$, we use the following classical variant of Scheffé's lemma: a sequence of non-negative random variables $(\xi_n)$ converges in $L^1(\P)$ to some random variable $\xi$ whenever $\lim_{n\to \infty}\E(\xi_n) =\E(\xi)$ and $\lim_{n\to \infty} \xi_n =  \xi$ in probability. Note that the usual Scheffé's lemma makes the stronger requirement $\xi_n \to \xi$ a.s., but using that from any extraction of $\xi_n$ one can find a subsequence converging almost surely to $\xi$, the result still holds.

 We denote by $\mathcal{T}_*$ the Galton-Watson subtree of $\mathcal{T}$ obtained by only keeping elements of $\mathcal{T}$ with outdegree $\ks$. Recall that $\#\{|u| = \ell : u \in \mathcal{T}_* \} = Z_*(\ell)$. We prove the convergence in probability of $ m_{\nu,q}^{-n}Z(n)$ to $W_*$ by decomposing $\mathcal{Z}_n$ at an intermediate generation $\ell$ as $|\mathcal{Z}_n| = |\mathcal{Z}_n^{\ell*}| + R_n$, where $\mathcal{Z}_n^{\ell*}$ is the point measure associated to individuals at generation $n$ in $\mathcal{T}$ with ancestors at generation $\ell$ that belong to $\mathcal{T}_*$. Using the branching property at generation $\ell$, we remark that $|\mathcal{Z}_n^{\ell*}|$ is the sum of $Z_*(\ell)$ independent copies of $|\mathcal{Z}_{n-\ell}|$ under law $\mathcal{P}_{\ell\delta_\ks}$, while $R_n$ is the sum of at most $\ks^\ell$ independent copies of $|\mathcal{Z}_{n-\ell}|$ starting from initial conditions such that $|\mathbf{t}|= \ell$ and $\mathbf{t}(\ks) < \ell$.

 As a result, for each fixed $\ell > 0$, we have  by Lemma~\ref{L:L1boundbis} that
 \[
   \lim_{n \to \infty} \mathcal{E}_\mathbf{0}\left( m_{\boldsymbol{\nu},q}^{-n}R_n \right) = 0.
 \]
 In particular, $m_{\boldsymbol{\nu},q}^{-n}R_n$ converges to $0$ in probability.

 We now compute the mean and variance of $|\mathcal{Z}_n^{\ell*}|$ conditionally on the first $\ell$ generations of the process. Using the consequence of the branching property described above, we have
 \[
   \mathcal{E}_\mathbf{0}\big(|\mathcal{Z}_n^{\ell*} | {\big |} \mathcal{T}_{|\ell}\big) = Z_*(\ell) \mathcal{E}_\mathbf{\ell \delta_{\ks}}\left(|\mathcal{Z}_{n-\ell}^{\ell*}|\right) \sim_{n \to \infty} m_{\boldsymbol{\nu},q}^n W_*(\ell) \quad \text{a.s.}
 \]
 by Lemma \ref{L:L1boundbis}. Similarly, we compute the conditional variance
 \begin{align*}
&   \mathcal{E}_\mathbf{0} \left( \left(|\mathcal{Z}_n^{\ell*}| - \mathcal{E}_\mathbf{0} \left( |\mathcal{Z}_n^{\ell*}| \middle|\mathcal{T}_{|\ell}\right) \right)^2 \big | \mathcal{T}_{|\ell}z\right) \\
   &=  Z_*(\ell) \mathcal{E}_\mathbf{\ell \delta_{\ks}}\left( \left(|\mathcal{Z}_{n-\ell}^{\ell*}| - \mathcal{E}_\mathbf{\ell \delta_{\ks}} \left( |\mathcal{Z}_{n-\ell}^{\ell*}| \right) \right)^2 \right) \\
   &\leq C_{\boldsymbol{\nu},q} Z_*(\ell) m_{\boldsymbol{\nu},q}^{2(n-\ell)}\left( \frac{m_{\boldsymbol{\nu},q}}{m_{*,q}} \right)^{2\ell},
 \end{align*}
 by Lemma~\ref{L:new}. Therefore, for all $\epsilon > 0$, applying a conditional Bienaymé-Chebyshev inequality, we obtain that for all $0 \leq \ell \leq n$
  \[
    \mathcal{P}_\mathbf{0}\left( m_{\boldsymbol{\nu},q}^{-n} \left| |\mathcal{Z}_n^{\ell*}| - \mathcal{E}_\mathbf{0}\left( |\mathcal{Z}_n^{\ell*}| \middle| \mathcal{T}_{|\ell} \right) \right| >\epsilon \right) \leq C_{\boldsymbol{\nu},q} \mathcal{E}_0(Z_*(\ell)) m_{*,q}^{-2\ell} \epsilon^{-2},
  \]
  and we observe that this bounds converges to $0$ as $\ell \to \infty$ uniformly in $n$.

  Let $\epsilon > 0$, we fix $\ell > 0$ large enough such that
  \[
    \mathcal{P}_\mathbf{0}(|W_*(\ell) - W_*|> \epsilon) < \epsilon \quad \text{and} \quad C_{\boldsymbol{\nu},q} \mathcal{E}_0(Z_*(\ell)) m_{*,q}^{-2\ell} \epsilon^{-2} \leq \epsilon,
  \]
  then $n \geq \ell$ large enough such that
  \[
    \mathcal{P}_\mathbf{0}\left( \left| \mathcal{E}_\mathbf{0}\left( m_{\boldsymbol{\nu},q}^{-n}|\mathcal{Z}_n^{\ell*}| \middle| \mathcal{T}_\ell \right) - W_*(\ell)\right| > \epsilon \right) \leq \epsilon \quad \text{and} \quad     \mathcal{P}_\mathbf{0}( m_{\boldsymbol{\nu},q}^{-n} R_n > \epsilon) \leq \epsilon,
  \]
  we have
  \[
    \mathcal{P}_0\left( |Z_n - W_*| > 4\epsilon \right) \leq 4\epsilon,
  \]
  which completes the proof.

\section{Proof of Theorem~\ref{T2}}
\label{sec:p2}

In this section, we will prove Theorem~\ref{T2};
let us briefly recall our approach using the notation we introduced.
We work from the viewpoint of multitype branching processes under $\mathcal P_{\delta_\ell}$  for some arbitrary fixed $\ell \geq 1$ with $\nu(\ell)>0$. This is equivalent to consider, under $\P_{\ell}$, the subpopulation generated by one of the $\ell$ individuals at the first generation.
We shall introduce a nonnegative martingale $M=(M_n)_{n\geq 1}$ starting from $M_0=1$, which is naturally related to the dynamics of the reinforced Galton-Watson process, and
 in particular vanishes as soon as $Z$ becomes extinct. Therefore $Z$ survives on the event that the terminal value
 $M_\infty$ is nonzero, and the latter happens with positive probability as soon as $M$ is uniformly integrable.

\subsection{A natural martingale}
\label{sec:mart}
We start by introducing some notation.
For every type $\mathbf t$, we set
$$m_{\mathbf t}\coloneqq \mathcal E_{\mathbf t}(Z(1)) = \sum_{j} j {\pi}_{\mathbf t}(j),$$
where, using the notation in Section 2, $\boldsymbol{\pi}_{\mathbf t}$ is the reproduction law of an individual with type $\mathbf t$.
More explicitly, for a non-zero type $\mathbf t \in \mathcal{M}_{\boldsymbol{\nu}}$,
we have
\begin{equation} \label{E:m_t}
m_{\mathbf t}= (1-q) m_\GW + \frac{q}{|\mathbf{t}|} \sum_{j=1}^{\ks} j{t}(j),
\end{equation}
where $m_\GW=\sum_{j} j \nu(j)$ is the mean reproduction number for the usual Galton-Watson process with reproduction law $\boldsymbol{\nu}$.
We further define, for ${u} \in \mathcal{T}$,
\begin{equation} \label{E:Phi}
\Phi({u}) \coloneqq \prod_{j=0}^{|{u}|-1} \frac{1}{m_{\mathbf{t}(p_j({u}))}},
\end{equation}
where $p_j({u})$ denotes the prefix of $u$ with length $j$ and thus represents the forebear of ${u}$  at generation $j$, with the convention that $\Phi(\varnothing)=1$.
In words, for any individual ${u}\in \mathcal{T}$, say at generation $|{u}|=k\geq 1$,
$1/\Phi({u})$ is the product of the mean reproduction numbers of the forebears of this individual.

If follows immediately from the definition of the mean $m_\mathbf{t}$ and the function $\Phi$ that the process $M = (M_n)_{n \geq 1}$ given by
\begin{equation}\label{E:martingale}
  M_n = \sum_{{u} \in \mathcal{T}, |{u}|=n} \Phi(u)
\end{equation}
is a martingale under $\mathcal P_{\mathbf t}$ for any initial type ${\mathbf t}$. Indeed, using the branching property of Lemma~\ref{L:branch} we have
\[
  \mathcal{E}_\mathbf{t} \left( M_{n+1} \middle| \mathcal{T}_{|n} \right) = \sum_{{u} \in \mathcal{T} : |{u}|=n} \frac{\Phi({u})}{m_{\mathbf{t}({u})}} \mathcal{E}_{\mathbf{t}({u})}(|\mathcal{Z}_1|) = M_n.
\]
In other words, $\Phi$ is mean-harmonic for the multitype branching process in the sense of \cite{BiK04}. The main purpose of this section is to study the asymptotic behavior of $M_n$ as $n \to \infty$.

\begin{proposition} \label{P:ui}
We assume that $q\ks < 1$, and consider an arbitrary $\ell \geq 1$ with $\nu(\ell)>0$.
\begin{enumerate}
  \item[(i)]  If \eqref{eqn:supercritical} holds, then the martingale $M$ is uniformly integrable under $\mathcal P_{\delta_\ell}$.
  \item[(ii)]  If
\begin{equation}\label{eqn:subcritical}
  \sum_{j =1}^{\ks} \frac{(1-q)j\nu(j)}{1-qj} < 1,
  \end{equation}
then
  the terminal value of the  martingale $M$ is  $M_{\infty}=0$, $\mathcal P_{\delta_\ell}$-a.s.
\end{enumerate}
\end{proposition}

As it was  already pointed out, Proposition~\ref{P:ui}(i) entails Theorem~\ref{T2}.
The rest of this section is devoted to the proof of  Proposition ~\ref{P:ui}.
The analysis relies on classical arguments involving a change of probability induced by the martingale $M$ and a decomposition of the branching  process along the spine. This leads us to investigate the asymptotic behavior of types along the spine, for which we will rely on classical results by Athreya and Karlin \cite{AK} and Janson \cite{Jan04} on P\'olya urns with random replacements.

\subsection{Spinal decomposition}
\label{sec:spine}
In this section, we introduce two distributions, $\bar{\mathcal P}_{\delta_\ell}$ and $\hat{\mathcal P}_{\delta_\ell}$, the first on the space of marked genealogical trees, and the second on the richer space of marked genealogical trees with a distinguished infinite branch called the spine.
The spinal decomposition then identifies $\bar{\mathcal P}_{\delta_\ell}$ as the projection of $\hat{\mathcal P}_{\delta_\ell}$.
In this direction, recall from Section~\ref{sec:const} that $\boldsymbol{\pi}_{\mathbf t}$ stands for the reproduction law of an individual with type $\mathbf t$.
For any probability distribution $\boldsymbol{\pi}$ on $\Z_+$ with a finite and  non-zero first moment, we also denote by $\hat{\boldsymbol{\pi}}$ the size-biased distribution of $\boldsymbol{\pi}$, defined by
\[
  \ \hat{{\pi}}(k) = \frac{k {\pi}(k)}{\sum_{j=0}^\infty j {\pi}(j)}, \qquad k\geq 1.
\]

To start with,   $\bar{\mathcal P}_{\delta_{\ell}}$ is  defined for every $n\geq 1$ by
\[
  \ \bar{\mathcal P}_{\delta_\ell}(A) = {\mathcal E}_{\delta_\ell} \left( M_n \indset{A}\right), \qquad  \ \forall A \in \mathcal{F}_n,
\]
where $\left(\mathcal{F}_n\right)_{n\geq 1}$ stands for the natural filtration on the space of genealogical trees (with types) induced by  generations.
We next construct another distribution involving a spine $\varsigma= (\varsigma(n))_{n\geq 0}$, where the latter is a distinguished line of descent, that is, a sequence of individuals such that for every $n\geq 0$,  $\varsigma(n+1)$ is a child of  $\varsigma(n)$. Specifically, $\varsigma(0)= \varnothing$ and  we let $\varsigma(0)$ reproduce according  to the size-biased reproduction law $\hat{ \boldsymbol{\pi}}_{\delta_\ell}$.
We then select an individual $\varsigma(1)$ uniformly at random amongst the, say, $k$ children of $\varsigma(0)$ which have all the type $\mathbf{t}(\varsigma(1)) = \delta_\ell + \delta_k$. Again we let $\varsigma(1)$ reproduce according  to the size-biased reproduction law $\hat{ \boldsymbol{\pi}}_{\mathbf{t}(\varsigma(1))}$, whereas each other individual ${u}$ in the sibling reproduce independently according to $ \boldsymbol{\pi}_{{u}}$.  And so on, and so forth, that is by induction at every generation, individuals reproduce independently of one from the other's, in such a way that a non-spine individual with type ${\mathbf t}$  reproduce according to the law $\boldsymbol{\pi}_{\mathbf t}$, and the spine particle $\varsigma(n)$, say of type ${\mathbf s}$, reproduces according to the law $\hat{\boldsymbol{\pi}}_{\mathbf s}$. The next individual of the spine $\varsigma(n+1)$ is chosen uniformly among the children of $\varsigma(n)$. The law of the resulting genealogical tree endowed with a spine is denoted by $\hat{\mathcal P}_{\delta_\ell}$.

The following result is then a consequence of \cite[Proposition 12.1 and Lemma 12.3]{BiK04}, the function $\Phi$ playing the role of the harmonic function $h$ in this article, and with $\hat{\mathcal{P}}$ corresponding to $\Q$ there. It is referred to as the spine decomposition of a multitype branching process.
\begin{proposition}[Spine decomposition of the reinforced Galton-Watson process] \label{P:spinedec}
For all $n \in \N$, we have
\[
  \left.\bar{\mathcal P}_{\delta_\ell}\right|_{\mathcal{F}_n} = \left.\hat{\mathcal P}_{\delta_\ell}\right|_{\mathcal{F}_n}.
\]
Moreover, for all $n \in \N$ and ${u} \in \mathcal U$ with $|u|=n$, we have
\[
 \hat{\mathcal P}_{\delta_\ell}(\varsigma_n = {u}| \mathcal{F}_n) = \frac{\Phi(u)\ind{{u} \in \mathcal{T}}}{M_n}.
\]
\end{proposition}

\begin{remark}
As usual, this spine decomposition result gives rise to a \emph{many-to-one} type lemma, yielding in particular to an alternative proof of \cite[Equation (2.1)]{BM}. We observe that
\begin{equation*}
  \mathcal{E}_\ell(|\mathcal{Z}_n|) = \hat{\mathcal{E}}_\ell\left(\frac{|\mathcal{Z}_n|}{M_n}\right) = \hat{\mathcal{E}}_\ell\left( \sum_{|u|=n} \frac{\ind{u = \varsigma(n)}}{\Phi(u)} \right)
  = \hat{\mathcal{E}}_\ell\left( \prod_{j=0}^{n-1} m_{\mathbf{t}(\varsigma(j))} \right).
\end{equation*}
Using the definition of $m_\mathbf{t}$, we conclude that $\mathcal{E}_\ell(|\mathcal{Z}_n|)$ can be computed as the mean of the product of the number of children sampled along a randomly selected line in the reinforced branching process.
\end{remark}

In the next section, we study in more details the reproduction law of the spine particle in terms of an urn model. It allows to describe the reinforced Galton-Watson process under law $\hat{\mathcal P}_{\delta_\ell}$ as a multitype branching process with immigration, whose asymptotic behaviour can be studied. Then, using a classical argument due to Durrett \cite[Theorem 4.3.5]{Dur19} (see also \cite[Theorem 3]{BiK04}), we are able to provide necessary and sufficient conditions for the uniform integrability of the martingale $M$, which we translate into Proposition~\ref{P:ui} in Section~\ref{sec:reinforcedPerpetuity}.

\subsection{Dynamics of the spine as a generalized P\'olya urn}
\label{sec:urns}

For $n\geq 0$, let $\xi_{n+1}$ denote the number of children of $\varsigma(n)$, the individual  on the spine at generation $n$. We also
 agree that $\xi_0\equiv \ell$ under $\hat{\mathcal P}_{\delta_\ell}$;
in particular, the type of  $\varsigma(n)$  is given by $\boldsymbol{\tau}_n \coloneqq  \sum_{j=0}^{n}\delta_{\xi_{j}}$.
From the construction of $ \hat{\mathcal P}_{\delta_\ell} $ in the preceding subsection, $\xi_1$ has the law $\hat{ \boldsymbol{\pi}}_{\delta_\ell}$, and we have for any $k\geq 1$ that
\[
  \hat{\mathcal P}_{\delta_\ell}(\xi_{k+1} = j\mid \xi_0,\xi_1, \ldots, \xi_k) = c_k j \left( (1 - q) \nu(j) + q \boldsymbol{\tau}_k(j)/(k+1) \right),\ j \geq 0,
\]
where $c_k> 0$ is the constant of normalization.

In this section, we shall first identify these dynamics as those of a generalized P\'olya urn. Next, we will determine the asymptotic behavior of the urn process by identifying the principal spectral elements of its mean replacement matrix, using classical results of Janson \cite{Jan04} in this field. This enables us to estimate the value of the mean-harmonic function $\Phi$ defined in \eqref{E:Phi} along the spine.

We denote first by $\mathcal C \coloneqq \{k \geq 1 : \nu(k) >0\}$ the support of the reproduction law. We think of any $k\in \mathcal C$ as a color,
and add a special color denoted by $\star$. We define an urn process with balls having colors in $\mathcal C \cup \{\star\}$ as follows. Imagine that a ball with color $k\in \mathcal C$ has activity $qk$, meaning that the probability that it is picked at some random drawing from the urn is proportional to $qk$, whereas a ball with color $\star$ has activity $(1-q) m_\GW$.

At the initial  time $n=0$, the urn contains one ball with color $\ell$ and one ball with color $\star$.
At each step $n\geq 2$, a ball is drawn at random in the urn with probability proportional to its activity. The ball is then replaced in the urn and two new balls are added to the urn. If the color of the sample ball is  $j \in \mathcal C$, then the first new ball has the color $j$ and the second the color $\star$. If the sampled ball has color $\star$, then the first new ball has the color $\star$ and the color of the second is sampled according to the law size-biased reproduction law $\hat{\boldsymbol{\nu}}$. Write $X_n$ the label of the non-$\star$ ball added to the urn at time $n$. We also set  $X_0=\ell$ as we initiate the urn with a ball with color $\ell$ and a ball with color $\star$.

\begin{lemma} \label{L:urn&spine}
\label{lem:urn} The sequence  $(X_n, n \geq 0)$ of colors added to the urn as  above has  the same law as $ (\xi_n, n \geq 0) $ under $\hat{\mathcal P}_{\delta_\ell}$.
\end{lemma}

\begin{proof}For all $n\geq 0$ and $j\in \mathcal C \cup \{\star\}$, we denote by $N_n(j)$ the number of balls with color $j$ in the urn after $n$ steps. It is plain
\[
  N_n(\star) = \sum_{j \in \mathcal C} N_n(j) = n+1.
\]
Moreover, for any $j \in \mathcal C$, we have
\[
  \P(X_{n+1} = j|X_0,\ldots,X_n)\\ = \frac{(1-q)m_\GW N_n(\star)\hat{\nu}(j)+  j qN_n(j)}{ (1-q)m_\GW N_n(\star)+q\sum_{i \in \mathcal C} i N_n(i) } ,
\]
since  $\{X_{n+1} = j\}$ is the event that at time $n+1$, either a ball labelled $j$ was sampled, or a ball labelled $\star$ and the extra ball added was labelled $j$.
Since  $j \nu(j) =   m_\GW \hat{\nu}(j)$, this  quantity can be rewritten as
\begin{align*}
  \P(X_{n+1} = j|X_0,\ldots,X_n) &= \frac{ (1-q)  (n+1)j\nu(j)+qjN_n(j) }{ (1-q)m_\GW (n+1)+q\sum_{i \in \mathcal C} i N_n(i) }\\
  &= \hat{\pi}_{\boldsymbol{\tau}_n}(j),
\end{align*}
where $\boldsymbol{\tau}_n = \sum_{j=0}^n \delta_{X_j}$. This proves that the dynamics of $X$ are identical to those of $\xi$ under law $\hat{\mathcal P}_{\delta_\ell}$.
\end{proof}

We next study the asymptotic behaviour of $X_n$ as $n \to \infty$ by applying general results of Athreya and Karlin \cite{AK} and Janson \cite{Jan04} on generalized P\' olya urns. Recall that $N_n(j)$ denotes the number of balls with color $j\in \mathcal C \cup \{\star\}$ in the urn after $n$ steps. By \cite[Section 4.2]{AK} or  \cite[Theorem 3.21]{Jan04}, there exists a constant $c > 0$ such that for any $j\in \mathcal C\cup\{\star\}$,
\[
  \lim_{n \to\infty}\frac{N_n(j)}{n} =c\lambda_1 {v}_1(j) \quad \text{a.s.,}
\]
where $\lambda_1$ is the leading eigenvalue and $\mathbf{{v}}_1=(v_1(j))$ an associated left-eigenvector of the  matrix $A=(A_{i,j})_{i,j\in  \mathcal{C} \cup \{\star\}}$ is defined for $i,j\in  \mathcal{C}$  by
\begin{equation}
  \label{eqn:R}
  A_{i,j} = q i \delta_{i,j}, \quad A_{i,\star} = q i, \quad A_{\star,j} = (1-q) j \nu(j), \quad A_{\star,\star} = (1-q)m_\GW.
\end{equation}
Roughly speaking, the matrix $A$ is mean replacement matrix re-weighted by activities.
Beware that Janson \cite{Jan04} rather uses the notation $A$ for the transposed of our matrix; in particular left-eigenvectors $\mathbf{{v}}$ in our setting correspond to right-eigenvectors in \cite{Jan04}, and vice-versa. The following spectral properties of $A$ are the key to the analysis.

\begin{lemma}
\label{lem:diagonalize}
The eigenvalues of the matrix $A$ defined in \eqref{eqn:R} are all simple, real and nonnegative. They are given by $0$ and the $\#\mathcal{C}$ positive solutions of the equation
\begin{equation}
 \label{eqn:defEigenvectors}
  \sum_{j=1}^\ks \frac{(1-q) j \nu(j)}{x - qj} = 1, \qquad \frac{x}{q}\in\R\backslash \mathcal{C}.
\end{equation}
Moreover, a left-eigenvector  $\mathbf{{v}}_\lambda$ associated to an eigenvalue $\lambda\geq 0$ is given by
$${v}_\lambda(i) = \frac{(1-q) i \nu(i)}{\lambda - qi}\quad \text{for $ i\in \mathcal C$, and } {v}_\lambda(\star) = 1,$$
and similarly a right-eigenvector  $\mathbf{{u}}_\lambda$ associated to an eigenvalue $\lambda\geq 0$ is given by
$${u}_\lambda(i) = \frac{qi}{\lambda - qi}\quad \text{for $ i\in \mathcal C$, and } {u}_\lambda(\star) = 1. $$
\end{lemma}

\begin{proof}
We first remark that for each consecutive elements $i,j$ of $\mathcal{C}$, the function
\[
  x \mapsto  \sum_{j=1}^\ks \frac{(1-q) j \nu(j)}{x - qj}, \qquad \frac{x}{q}\in\R\backslash \mathcal{C},
\]
is decreasing from $\infty$ to $-\infty$ on the interval $(qi,qj)$. Additionally, this function is also decreasing on $(-\infty,q\min \mathcal{C})$ while staying non-positive, and is decreasing on $(q \ks,\infty)$ from $\infty$ to $0$. Consequently, there are exactly $\# \mathcal{C}$ roots to the equation \eqref{eqn:defEigenvectors}, all being positive.

We then consider any solution, say $\lambda>0$,  to \eqref{eqn:defEigenvectors}
and  check by immediate computations that
 $(\lambda,\mathbf{{v}}_\lambda)$ as defined in the statement are eigenvalues and associated left-eigenvectors for the matrix $A$.
 Indeed,  we have first for any $j \in \mathcal{C}$,
\[
  \sum_{i \in \mathcal{C} \cup \{\star\}} {v}_\lambda(i) A_{i,j} = qj  \frac{(1-q) j \nu(j)}{\lambda - qj} +(1-q) j \nu(j) = \lambda {v}_\lambda(j),
\]
and then  for $j = \star$,
\begin{align*}
  \sum_{i \in \mathcal{C} \cup \{\star\}} {v}_\lambda(i) A_{i,\star} &= (1-q) m_\GW +\sum_{i=1}^\ks   qi   \frac{(1-q) i \nu(i)}{\lambda - qi} \\
  &= (1-q) \sum_{i=1}^\ks \left(  i \nu(i) + qi  \frac{i \nu(i)}{\lambda-qi} \right)\\
  & = (1-q) \sum_{i=1}^\ks {i \nu(i)} \frac{\lambda}{\lambda - qi}  \\
  &= \lambda ,
\end{align*}
where the ultimate equality uses that $\lambda$ solves \eqref{eqn:defEigenvectors}.

We check by similar calculations for $(0,\mathbf{{v}}_0)$ that
\[
  \sum_{i \in \mathcal{C} \cup \{\star\}} {v}_0(i) A_{i,j} =0 \qquad \text{for all } j\in \mathcal C \cup\{\star\}. \]
 Finally, we verify in the same way that the $\mathbf{{u}}_\lambda$ are also right-eigenvectors.

As a conclusion, we have found $\#\mathcal{C}+1$ different real eigenvalues of the $(\#\mathcal{C}+1) \times (\#\mathcal{C}+1)$ matrix $A$. They are hence all simple and there exist no further eigenvalues.
\end{proof}

We order the positive eigenvalues of the mean replacement matrix $A$ in the decreasing order, $\lambda_1>\lambda_2> \ldots > \lambda_{\#\mathcal{C}}$
and then write simply $\mathbf{{v}}_i=\mathbf{{v}}_{\lambda_i}$ and $\mathbf{{u}}_i=\mathbf{{u}}_{\lambda_i}$ for the corresponding left and right eigenvectors given in Lemma~\ref{lem:diagonalize}. Note that $\mathbf{{u}}_1$  and $\mathbf{{v}}_1$ have positive coordinates (as it should be expected from Perron-Frobenius' theorem), and that we did not impose the usual normalization that their scalar product should be $1$, for the sake of simplicity. It is a well-known fact  that the first order asymptotic behaviour of $N_n$ is determined by  $\lambda_1$ and $\mathbf{{v}}_1$, whereas the fluctuations depend on the sign of $\lambda_2 - \lambda_1/2$. More precisely, we first apply \cite[Theorem 3.21]{Jan04}, to obtain the asymptotic behavior of the number  $N_n(j)$ of balls with color $j$ after $n$ steps as $n \to \infty$.

\begin{lemma} \label{L:Perron} We have for all $j \in \mathcal C$ that
\[
  \lim_{n \to \infty} \frac{N_n(j)}{n} = v_1(j)= \frac{(1-q) j \nu(j)}{\lambda_1 - qj}, \qquad \hat{\mathcal P}_{\delta_\ell}\text{-a.s.}
\]
\end{lemma}

\begin{proof} The  matrix $A$ is irreducible and aperiodic, and
Lemma~\ref{lem:diagonalize} entails that the conditions (A1--6) of \cite{Jan04} are satisfied. By \cite[Theorem 3.21]{Jan04}, there exists $c > 0$ such that for all $j \in \mathcal{C} \cup \{\star\}$
\[
  \lim_{n \to \infty} \frac{N_n(j)}{n} = c \lambda_1 {v}_1(j) \quad \text{a.s.}
\]
Moreover, as $N_n(\star) = n+1$ a.s., we have $N_n(\star)/n \to 1$ a.s. As $ {v}_1(\star) = 1$,  we conclude that $c = 1/\lambda_1$, which completes the proof.
\end{proof}

\begin{corollary}\label{C:spinecv}
Assume that $q\ks < 1$.
\begin{itemize}
  \item[(i)] The principal eigenvalue $\lambda_1>1$ if and only if
\eqref{eqn:supercritical} holds.
In that case, the series $\sum_{k=0}^{\infty} \Phi(\varsigma(n))$ converges a.s.

 \item[(ii)] The principal eigenvalue $\lambda_1<1$ if and only if \eqref{eqn:subcritical} holds.
In that case, we have $\lim_{n\to \infty}\Phi(\varsigma(n)) = \infty$ a.s. \end{itemize}
\end{corollary}

\begin{proof}
The equivalences
$$\lambda_1 >1 \Longleftrightarrow \sum_{j =1}^{\ks} \frac{(1-q)j\nu(j)}{1-qj} > 1,$$
and
$$\lambda_1 <1 \Longleftrightarrow \sum_{j =1}^{\ks} \frac{(1-q)j\nu(j)}{1-qj} <1,$$
should be plain from Lemma~\ref{lem:diagonalize}. Indeed, the only eigenvalue greater than $q \ks < 1$ is the Perron-Frobenius principal eigenvalue $\lambda_1$, and we have seen  that the function  $x \mapsto \sum_{j =1}^{\ks} \frac{(1-q)j\nu(j)}{1-qj}$ decreases on $[1,\infty)$.

Next, recall that the type $\boldsymbol{\tau}_n$ of the individual $\varsigma(n)$ is  $\boldsymbol{\tau}_n = \sum_{j=0}^{n}\delta_{\xi_{j}}$,
where $\xi_k$ stands the number of children of $\varsigma(k-1)$. We know moreover from Lemma~\ref{L:urn&spine} that the sequence $(\xi_k)_{k\geq 0}$ has the same distribution under $\hat{\mathcal P}_{\delta_\ell}$ as the sequence in $\mathcal C$ of the colors of the balls added to the urn at each step.
We deduce from Lemma~\ref{L:Perron} that $\hat{\mathcal P}_{\delta_\ell}$-a.s., there is the convergence
$$\lim_{n\to \infty} \frac{ {\tau}_n (j)}{ |\boldsymbol{\tau}_n |}=  v_1(j), \qquad \text{for all }j\in \mathcal C.$$
We then get from \eqref{E:m_t} that
\begin{align*}\lim_{n\to \infty} m_{\boldsymbol{\tau}_n}
&= (1-q) \sum_{j=1}^{\ks} j \nu(j) + q \sum_{j=1}^{\ks} \frac{(1-q) j^2 \nu(j)}{ \lambda_1 - qj}\\
&= (1-q) \sum_{j =1}^{\ks} \frac{  \lambda_1 j\nu(j)}{\lambda_1-qj} \\
&= \lambda_1,
\end{align*}
where the last equality is due to the fact that $\lambda_1$ solves \eqref{eqn:defEigenvectors}.
Last, we deduce from the very definition \eqref{E:Phi} of $\Phi$ that as $n\to \infty$,
$$\log \Phi(\varsigma(n)) \sim -n \log \lambda_1.$$
As a consequence,   the series $\sum_{k=0}^{\infty} \Phi(\varsigma(n))$ converges a.s. whenever $\lambda_1 >1$, whereas $\lim_{n\to \infty} \Phi(\varsigma(n))= \infty$ a.s. whenever $\lambda_1<1$.
\end{proof}

We have now all the ingredients needed to establish  Proposition~\ref{P:ui}.

\subsection{Proof of Proposition~\ref{P:ui}}
\label{sec:reinforcedPerpetuity}
The starting point of the proof  is the following well-known observation due to Durrett \cite[Theorem 4.3.5]{Dur19} (see also \cite[Theorem 12.1]{BiK04}).
Using that $M_0 = 1$, and writing $M_\infty = \lim_{n \to \infty} M_n$, we have
\begin{equation}
  \label{eqn:durretFormula}
  \mathcal{E}_\ell (M_\infty)  = \bar{\mathcal{P}}_{\delta_\ell} (M_\infty < \infty).
\end{equation}
Recall that the tilted law $\bar{\mathcal{P}}_{\delta_\ell}$ has been introduce in Section 5.2, and that the spine decomposition of the reinforced Galton-Watson process has been described in Proposition~\ref{P:spinedec}

It is worth noting that $M$ is a non-negative martingale under $\mathcal{P}_{\delta_\ell}$, and $1/M$ a non-negative super-martingale under $\bar{\mathcal{P}}_{\delta_\ell}$, hence the convergence of $M$ is immediate under the original and tilted laws. As a result, to prove Proposition~\ref{P:ui}(i), it is enough to study the convergence of $M$ under the richer law $\hat{\mathcal P}_{\delta_\ell}$ by Proposition~\ref{P:spinedec}. Specifically, almost sure finiteness of $M_\infty$ under $\hat{\mathcal{P}}_{\delta_\ell}$ implies uniform integrability under $\mathcal{P}_{\delta_\ell}$, while if $M_\infty = \infty$ $\hat{\mathcal{P}}_{\delta_{\ell}}$-a.s., then $M_\infty = 0$ $\mathcal{P}_{\delta_\ell}$-a.s.

(i)  We denote by $\mathcal{Y} = \sigma(\xi_k, k \geq 0)$ the sigma-field of the information on the number of children of all spine particles. The spinal decomposition yields

\[
  \hat{\mathcal E}_{\ell}(M_n|\mathcal{Y}) = \Phi(\varsigma(n)) + \sum_{k=0}^{n-1}(\xi_k-1)\Phi (\varsigma(k)), \qquad \hat{\mathcal P}_{\delta_\ell}\text{-a.s.},
\]
using that the individual  $\varsigma(k-1)$ at generation $k-1$ on the spine has $\xi_k$ children, one of those chosen as the individual on the spine at generation $k$ and the $\xi_k-1$
other evolve independently according to the law $\mathcal P_{\mathbf{t}(\varsigma(k))}$.
Since $\xi_k-1\leq \ks$, we have \textit{a fortiori} that
\[
  \hat{\mathcal E}_{\ell}(M_n|\mathcal{Y}) \leq \ks \sum_{k=0}^\infty \Phi (\varsigma(k))), \qquad \hat{\mathcal P}_{\delta_\ell}\text{-a.s.}
\]
We can now conclude from Corollary~\ref{C:spinecv} and the conditional version of the Fatou lemma that if
\[
  \sum_{j =1}^{\ks} \frac{(1-q)j\nu(j)}{1-qj} > 1,
\]
then $ \liminf_{n \to \infty} M_n < \infty$, $\hat{\mathcal P}_{\delta_\ell}$-a.s. As a result, by Proposition~\ref{P:spinedec} and \eqref{eqn:durretFormula} we have $\mathcal{E}_\ell (M_\infty) =1$, and by the Scheff\'e lemma,
$M$ is uniformly integrable under $\mathcal{P}_{\delta_\ell}$.

(ii) The simple observation that $M_n \geq \Phi(\varsigma(n))$ combined with Corollary~\ref{C:spinecv}(ii) enables us to conclude that  $\hat{\mathcal{P}}_{\delta_\ell}(M_\infty = \infty) =1$ whenever \eqref{eqn:subcritical} holds. Using again Proposition~\ref{P:spinedec} and \eqref{eqn:durretFormula}, the proof is now complete.

\subsection{Short discussion of the critical case}
The proof of Proposition~\ref{P:ui} in the preceding section relies on the study of
the quantity  $\sum_{k = 0}^\infty \Phi(\varsigma(n))$, which can be thought of as a particular case of \emph{reinforced perpetuity}.
We have seen above  that this series converges when $\lambda_1>1$ whereas $\lim_{n\to\infty} \Phi(\varsigma(n))=\infty$ when $\lambda_1<1$.
We now conclude this work by discussing briefly and a bit informally the critical case when the principal eigenvalue is $\lambda_1=1$, which presents interesting complexities. Recall from Corollary~\ref{C:spinecv} that this is equivalent to
$$  \sum_{j =1}^{\ks} \frac{(1-q)j\nu(j)}{1-qj} = 1.$$

We  shall need the following result (that however does not requires $\lambda_1=1$) about the fluctuations of the convergence in Lemma~\ref{L:Perron}.
\begin{lemma}
\label{lem:secondTerm} \begin{itemize}
\item[(i)] \textrm{(Heavy urn)}
If $\lambda_2>\lambda_1/2$, then there exists a square integrable non-degenerate random variable $W_2$ such that
for all $j \in \mathcal{C}$, \[
  \lim_{n \to \infty} n^{-\lambda_2/\lambda_1} (N_n(j) - n \lambda_1 {v}_{\lambda_1}(j)) = W_2 v_2(j), \quad \hat{\mathcal P}_{\delta_\ell}\text{-a.s.}
\]
\item[(ii)] \textrm{(Light urn)}
If $\lambda_2<\lambda_1/2$, then there is the joint convergence in distribution for all $j\in \mathcal{C}$
\[
  \lim_{n \to \infty} n^{-1/2} \left(N_{[nt]}(j) - nt \lambda_1 {v}_{\lambda_1}(j))\right)_{t\geq 0} = (G_t(j))_{t\geq 0}, \quad \text{in law},
\]
where $\mathbf G=(G(j))_{j\in \mathcal C}$ is some continuous centered $\#\mathcal C$-dimensional Gaussian process .
\end{itemize}
\end{lemma}

\begin{proof}
This is a direct application of \cite[Theorem 3.24 and Theorem 3.31]{Jan04}, using that $\lambda_2$ is a simple eigenvalue with left-eigenvector $\mathbf{{v}}_2$, and that $A$ has no non-real eigenvalues.
\end{proof}

In our setting, we now see that if $\lambda_1=1$ and $\lambda_2>1/2$, then as $n\to \infty$,
$$m_{\boldsymbol{\tau}_n} =1+ n^{\lambda_2-1} W_2\sum_{j=1}^\ks \frac{(1-q)j^2\nu(j)}{\lambda_2-qj} + o(n^{\lambda_2-1}), \qquad \text{ $ \hat{\mathcal P}_{\delta_\ell}$-a.s.}$$
Moreover, we remark that
\begin{align*}
  \sum_{j=1}^\ks \frac{(1-q)j^2 \nu(j)}{\lambda_2- qj} &= \frac{1}{q} \left( \sum_{j=1}^\ks \frac{(1-q)j(qj-\lambda_2)\nu(j)}{\lambda_2 - q j} + \lambda_2 \sum_{j=1}^\ks \frac{(1-q)j \nu(j)}{\lambda_2 - q  j} \right)\\
  &= \frac{\lambda_2- (1-q)m_\GW}{q},
\end{align*}
using that $\lambda_2$ solves \eqref{eqn:defEigenvectors}.
As a consequence, on the event
$$\left \{W_2 (\lambda_2- (1-q)m_\GW)>0\right \},$$
the series $\sum_{k = 0}^\infty \Phi(\varsigma(n))$ converges and
the terminal value $M_{\infty}<\infty$.
At the opposite,  on the event
$$\left \{W_2(\lambda_2- (1-q)m_\GW)<0 \right \},$$
the series $\sum_{k = 0}^\infty \Phi(\varsigma(n))$  diverges.

We remark that $W_2$ is a random variable that is positive, respectively negative, with positive probability. Indeed, using that $W_2$ is obtained as the limit of a martingale of the form $a_n^{-1}\sum_{j=1}^n (v_2(X_j) + v_2(\star))$, we observe that
$$\E_\ell(W_2|X_1,\ldots,X_n) \propto v_2(\ell) + \sum_{j=1}^n v_2(X_j) + (n+1) v_2(\star),$$
can be positive or negative depending on the values of $X_1,\ldots,X_n$. Therefore, assuming that $\lambda_2- (1-q)m_\GW \neq 0$, the reinforced perpetuity  converges, respectively diverges, with positive probability. In particular the martingale $M$  converges to a non-degenerate random variable, with positive probability under law $\mathcal{P}_{\delta_\ell}$,  while being non-uniformly integrable.

In the light case when $\lambda_1=1$ and $\lambda_2<1/2$, $n^{1/2}(m_{\boldsymbol{\tau}_n}-1)$ rather converges in distribution to some centered Gaussian variable.
This suggests that in most (if not all) cases, one should have $\limsup_{n\to\infty} \Phi(\varsigma(n))=\infty$.
This would imply that the terminal value $M_{\infty}=\infty$, $\hat{\mathcal P}_{\delta_\ell}$-a.s., and therefore also $M_{\infty}=0$,  ${\mathcal P}_{\delta_\ell}$-a.s.

The possible behaviors of critical reinforced perpetuity hence appear richer than for regular perpetuities. This also suggests that there exist reinforced branching processes such that \eqref{eqn:subcritical} holds but the process survives with positive probability.

\paragraph{Acknowledgements.}
The authors wish to thank the anonymous referees for their comments that helped improved previous versions of this article.

\end{document}